\documentclass[11pt,a4paper]{article}
\usepackage[german,english]{babel}
\usepackage{latexsym}
\usepackage{amsmath, amsfonts}
\usepackage{amssymb}	
\usepackage{amsthm}	
\usepackage{ifthen}
\usepackage{hyperref} 
\usepackage{graphicx} 
\usepackage[arrow,matrix]{xy} 
\usepackage{verbatim} 

\author{Michael Klotz\footnote{Technische Universit\"at Darmstadt, Schlo\ss gartenstra\ss e 7, D-64289 Darmstadt, Deutschland,\newline klotz@mathematik.tu-darmstadt.de}}
\date{}
\title{Subspaces and Quotients of Banach Symmetric Spaces}

\theoremstyle{definition} 
\newtheorem{definition}{Definition}[section]
\theoremstyle{plain} 
\newtheorem{theorem}[definition]{Theorem}	
\newtheorem{proposition}[definition]{Proposition}	
\newtheorem{lemma}[definition]{Lemma}
\newtheorem{corollary}[definition]{Corollary}	
\theoremstyle{definition}
\newtheorem{example}[definition]{Example}
\newtheorem{remark}[definition]{Remark}
\newtheorem{problem}[definition]{Problem}
\newenvironment{acknowledgements}{\section*{Acknowledgements}}{}

\newcommand{\temp}{0} 
	{\renewcommand{\temp}{\arraystretch} \renewcommand{\arraystretch}{#1}}
	{\renewcommand{\arraystretch}{\temp}}

\newcommand{\NN}{\mathbb N}

\newcommand{\RR}{\mathbb R}

\newcommand{\eins}{{\bf 1}}

\newcommand{\calA}{{\cal A}}

\newcommand{\calU}{{\cal U}}

\DeclareMathOperator{\ad}{ad}
\DeclareMathOperator{\Ad}{Ad}

\DeclareMathOperator{\Aut}{Aut}

\DeclareMathOperator{\Der}{Der}

\DeclareMathOperator{\Exp}{Exp}
\DeclareMathOperator{\ev}{ev}

\DeclareMathOperator{\flow}{Fl}
\DeclareMathOperator{\Fr}{Fr}
\DeclareMathOperator{\GL}{GL}

\DeclareMathOperator{\id}{id}
\DeclareMathOperator{\im}{im}
\DeclareMathOperator{\Inn}{Inn}

\DeclareMathOperator{\Lts}{Lts}

\DeclareMathOperator{\Sym}{Sym}

\newcommand{\gl}{\mathfrak{gl}}
\newcommand{\calLts}{\mathfrak{Lts}}

\begin{document}
\maketitle
%
%
%
%
\begin{abstract}
	We derive a necessary and sufficient condition for the existence of symmetric space structures on quotients of Banach symmetric spaces. Along the way, we investigate the different kinds of reflection subspaces and their Lie triple systems.\\[-0.5\baselineskip]
				
	
\end{abstract}
%
%
%
\section{Introduction}
\label{introduction}
A \emph{symmetric space} in the sense of O.~Loos (cf.\ \cite{Loo69}) is a smooth manifold $M$ endowed with a smooth multiplication map $\mu\colon M\times M\rightarrow M$ such that each left multiplication map $\mu_x:=\mu(x,\cdot)$ (with $x\in M$) is an involutive automorphism of $(M,\mu)$ with isolated fixed\linebreak point $x$.

Some basic material on infinite-dimensional symmetric spaces can be found in \cite{Nee02Cartan} and \cite{Ber08}. In \cite{Klo09b}, the author starts working towards a Lie theory of symmetric spaces modelled on Banach spaces. In particular, there are presented an integrability theorem for morphisms of Lie triple systems and the result that the automorphism group of a connected symmetric space $M$ is a Banach--Lie group acting smoothly and transitively on $M$.
The purpose of this paper is to continue on this path.

A \emph{reflection subspace} $N$ of a symmetric space $M$ is a subset that is stable under multiplication. If a reflection subspace is endowed with the structure of a symmetric space such that the inclusion map $\iota\colon N\rightarrow M$ is smooth and its tangent map induces at each $x\in N$ a topological embedding $T_x\iota$, then we call it an \emph{integral subspace}. In \cite{Loo69}, O.~Loos shows that, in the finite-dimensional case, an integral subspace of a pointed symmetric space $(M,b)$ induces a closed triple subsystem of $\Lts(M,b)$ and conversely, given any closed triple subsystem $\mathfrak{n}\leq \Lts(M,b)$, there is a unique connected integral subspace $(N,b)$ with Lie triple system $\mathfrak{n}$. We shall see that this one-to-one correspondence carries over to the Banach case.
If $\mathfrak{n}$ is separable, we further have $\mathfrak{n}=\{x\in\Lts(M,b)\colon \Exp_{(M,b)}(\RR x)\subseteq N\}$.

If a reflection subspace $N\leq M$ is a submanifold, then we call it a \emph{symmetric subspace}. We shall see that a symmetric subspace is just an integral subspace whose topology is induced by $M$. A natural chart of $N$ at $b$ can be obtained as a restriction of an exponential chart of $(M,b)$. The closed triple subsystem $\mathfrak{n}\leq \Lts(M,b)$ induced
by $(N,b)$ satisfies\linebreak $\mathfrak{n}=\{x\in\Lts(M,b)\colon \Exp_{(M,b)}(\RR x)\subseteq N\}$ even if it is not separable.
As a consequence, we observe that preimages of symmetric subspaces under morphisms are symmetric subspaces. We further show that a connected integral subspace $(N,b)$ of $(M,b)$ whose Lie triple system $\mathfrak{n}\leq \Lts(M,b)$ splits as a Banach space is a symmetric subspace if and only if there exists a 0-neighborhood $W\subseteq F$ in the complement $F$ of $\mathfrak{n}$ such that $N\cap\Exp_{(M,b)}(W)=\{b\}$.

In the finite-dimensional case, O.~Loos shows that a closed reflection subspace can be naturally turned into a symmetric subspace. Similarly to the matter of turning a closed subgroup of a Lie group into a Lie subgroup, that can not be generalized to the Banach case, but we can still assign Lie triple systems to closed reflection subspaces: Given a closed reflection subspace $(N,b)\leq (M,b)$, the set $\mathfrak{n}:=\{x\in\Lts(M,b)\colon \Exp_{(M,b)}(\RR x)\subseteq N\}$ is a closed subtriple system of $\Lts(M,b)$. To see this, we consider the Trotter product formula and the commutator formula for Lie groups and find comparable formulas for symmetric spaces.

An equivalence relation $R\subseteq M\times M$ on a symmetric space $M$ that is a reflection subspace of $M\times M$ is called a \emph{congruence relation}. If $M$ is connected, then $R$ is determined by each of its equivalence classes, which we then call \emph{normal reflection subspaces} of $M$.
Given a closed connected normal reflection subspace $N\unlhd M$ with associated congruence relation $R\leq M\times M$ on $M$, we show that the quotient $M/R$ can be made a symmetric space such that the natural quotient map $\pi\colon M\rightarrow M/R$ is a morphism of symmetric spaces and a ``weak'' submersion\footnote{By the term \emph{``weak'' submersion} we mean that the tangent map of $\pi$ induces at each $x\in M$ a linear quotient map $T_x\pi$.} if and only if $N$ is a symmetric subspace of $M$. The corresponding assertion in the finite-dimensional case can be found in \cite{Loo69}.
Along the way, we show that the Lie triple system $\mathfrak{n}\leq\Lts(M,b)$ of a closed normal reflection subspace $(N,b)\unlhd (M,b)$ is a (closed) ideal and that, given a closed ideal $\mathfrak{n}\unlhd\Lts(M,b)$, the corresponding connected integral reflection subspace $(N,b)\leq (M,b)$ is normal.

The guiding philosophy of our work is that a connected symmetric space actually is a Banach homogeneous space: It can be identified with the quotient of its automorphism group by a stabilizer subgroup (cf.\ \cite{Klo09b}). This is based on the theorem that its automorphism group is a Banach--Lie group. Therefore, we study symmetric Lie groups, symmetric pairs and the functor $\Sym$ that assigns to a symmetric pair its quotient symmetric space.

We can apply the results of this paper to give an integrability criterion for Lie triple systems. In the context of Banach--Lie groups, it is a well-known result that a Banach--Lie algebra $\mathfrak{g}$ is integrable if and only if its period group $\Pi(\mathfrak{g})$ (which is a subgroup of the center of $\mathfrak{g})$ is discrete (cf.\ \cite{EK64} and \cite{GN03}). In \cite{Klo10Integrability}, the author shows a similar statement for Lie triple systems.
\tableofcontents

%
%
%
\section{Basic Concepts and Notation} \label{sec:basicConcepts}
%
%
\subsection{Terminology for Submanifolds and Lie Subgroups}
\label{sec:terminologySub}
A subset $N$ of a smooth Banach manifold $M$ is called a \emph{local submanifold at $x\in N$} if there exists a chart $\varphi\colon U\rightarrow V\subseteq E$ of $M$ at $x$ and a closed subspace $F$ of $E$ such that $\varphi(U\cap N)=V\cap F$. If $N$ is a local submanifold at each $x\in N$, then it is called a \emph{submanifold} and we obtain charts $\varphi|_{U\cap N}^{V\cap F}$ for $N$ that define on $N$ the structure of a manifold, which is compatible with the subspace topology. If each subspace $F$ splits in $E$ as a Banach space, then $N$ is called a \emph{split submanifold}.
Note that, for each $x\in N$, the tangent map $T_x\iota$ of the inclusion map $\iota\colon N\hookrightarrow M$ is a (closed) topological embedding. A submanifold $N$ splits if and only if the inclusion map $\iota$ is an immersion, i.e., if for each $x\in N$, the image of the topological embedding $T_x\iota$ splits as a Banach space.

A subgroup $H$ of a Banach--Lie group $G$ is called a \emph{Lie subgroup} if it is a submanifold of $G$. It is called a \emph{split Lie subgroup} if it is even a split submanifold.
Every Lie subgroup is closed. Given a Lie subgroup $H\leq G$ with inclusion map $\iota\colon H\hookrightarrow G$, the Lie algebra $L(H)$ of $H$ can be identified with $\mathfrak{h}:=L(\iota)(L(H))$, which is given by $\mathfrak{h}=\{x\in L(G)\colon \exp_G(\RR x)\subseteq H\}$. There exists an open $0$-neighborhood $V\subseteq L(G)$ such that $\exp_G|_V$ is a local diffeomorphism onto an open subset of $M$ and $\exp_G(V\cap \mathfrak{h})=\exp_G(V)\cap H$ (cf.\ \cite[Th.~IV.3.3]{Nee06} or \cite[Prop.~3.4]{Hof75}).

An \emph{integral subgroup} of $G$ is a subgroup $H\leq G$ endowed with a Banach--Lie group structure such that the inclusion map $\iota\colon H\rightarrow G$ is smooth and $L(\iota)$ is a topological embedding.
Given some closed subalgebra $\mathfrak{h}$ of the Lie algebra $L(G)$ of $G$, then the subgroup $H:=\langle\exp_G(\mathfrak{h})\rangle$ carries a unique structure of a connected integral subgroup of $G$ with Lie algebra $\mathfrak{h}$ (cf.\ \cite[Satz~12.3]{Mai62}).
If $\mathfrak{h}$ is a closed ideal in $L(G)$, then the subgroup $H$ is normal in the connected component $G_0$ of the identity (cf.\ \cite[Satz~12.6]{Mai62}). 
If the Lie algebra of an integral subgroup $H$ of $G$ is separable, i.e., if it contains a countable dense subset, then we have $L(\iota)(L(H))=\{x\in L(G)\colon \exp_G(\RR x)\subseteq H\}$ (cf.\ \cite[Th.~IV.4.14]{Nee06}).
Note that an integral subgroup $H\leq G$ that is compatible with the subspace topology is a Lie subgroup.
%

%
%
%
%
\subsection{Lie Triple Systems and Symmetric Spaces} \label{sec:LtsAndSymSpace}
A \emph{Lie triple system} (cf.\ \cite{Loo69}) is a Banach space $\mathfrak{m}$ endowed with a continuous trilinear map\linebreak $[\cdot,\cdot,\cdot]\colon \mathfrak{m}^3\rightarrow\mathfrak{m}$ that satisfies $[x,x,y]=0$ and $[x,y,z] + [y,z,x] + [z,x,y] = 0$ as well as
\[[x,y,[u,v,w]] \ =\ [[x,y,u],v,w]+[u,[x,y,v],w]+[u,v,[x,y,w]]\]
for all $x,y,z,u,v,w\in \mathfrak{m}$. Continuous linear maps between Lie triple systems that are compatible with the Lie brackets are called \emph{morphisms}. A subspace $\mathfrak{n}$ of $\mathfrak{m}$ is called a \emph{triple subsystem} (denoted by $\mathfrak{n}\leq \mathfrak{m}$) if it is stable under the Lie bracket. If it satisfies the stronger condition $[\mathfrak{n},\mathfrak{m},\mathfrak{m}]\subseteq \mathfrak{n}$, then it is called an \emph{ideal} and we write $\mathfrak{n} \unlhd \mathfrak{m}$. An ideal $\mathfrak{n}\unlhd\mathfrak{m}$ automatically satisfies also the conditions $[\mathfrak{m},\mathfrak{n},\mathfrak{m}]\subseteq \mathfrak{n}$ and $[\mathfrak{m},\mathfrak{m},\mathfrak{n}]\subseteq \mathfrak{n}$.

A \emph{reflection space} (cf.\ \cite{Loo67a,Loo67b}) is a set $M$ endowed with a multiplication map $\mu\colon M\times M\rightarrow M$, $(x,y)\mapsto x\cdot y$, such that each left multiplication map $\mu_x:=\mu(x,\cdot)$ (with $x\in M$) is an involutive automorphism of $(M,\mu)$ with fixed point $x$. A subset of $M$ that is stable under $\mu$ is called a \emph{reflection subspace}. Given a subset $S\subseteq M$, we denote by $\langle S\rangle\leq M$ the \emph{generated reflection subspace}, i.e., the smallest reflection subspace of $M$ that contains $S$.
A \emph{symmetric space} is a smooth Banach manifold $M$ endowed with a smooth multiplication map $\mu\colon M\times M\rightarrow M$ such that $(M,\mu)$ is a reflection space for which each $x\in M$ is an isolated fixed point of the symmetry $\mu_x$. Maps between reflection spaces that are compatible with multiplication are called \emph{homomorphisms}. Concerning symmetric spaces, we refer to smooth homomorphisms as \emph{morphisms}.
If there is no confusion, we usually denote a reflection space (resp., symmetric space) simply by $M$ instead by $(M,\mu)$.

The following facts about symmetric spaces can be essentially found in \cite{Klo09b}, which is partially due to \cite{Ber08} and \cite{Nee02Cartan}.
The tangent bundle $TM$ endowed with the multiplication $T\mu$ is a symmetric space. In each tangent space $T_xM$ (with $x\in M$), the product satisfies $v\cdot w =2v - w$. A smooth vector field on $M$ is called a \emph{derivation} if it is a morphism of symmetric spaces. Note that every derivation is a complete vector field. The set $\Der(M)$ of all derivations is a Lie subalgebra of the Lie algebra of all smooth vector fields on $M$.\footnote{Here, the term \emph{Lie algebra} does not include a topological structure.} Given a distinguished point $b\in M$, called the \emph{base point}, the symmetry $\mu_b$ induces an involutive automorphism $(\mu_b)_\ast$ of $\Der(M)$ given by $(\mu_b)_\ast(\xi):=T\mu_b\circ\xi\circ\mu_b$. The (+1)-eigenspace $\Der(M)_+$ of $(\mu_b)_\ast$ is a subalgebra of $\Der(M)$ and coincides with the kernel of the evaluation map $\ev_b\colon\Der(M)\rightarrow T_bM$, $\xi\mapsto \xi(b)$. The (-1)-eigenspace $\Der(M)_-$ of $(\mu_b)_\ast$ is stable under the triple bracket $[\cdot,\cdot,\cdot]:=[[\cdot,\cdot],\cdot]$. Via the evaluation isomorphism $\ev_b|_{\Der(M)_-}\colon\Der(M)_-\rightarrow T_bM$ of vector spaces, the tangent space $T_bM$ can be equipped with that triple bracket. It becomes a \emph{Lie triple system} that we denote by $\Lts(M,b)$. Assigning to each morphism of pointed symmetric spaces its tangent map at the base point, we obtain a covariant functor $\Lts$ (called the \emph{Lie functor}) from the category of pointed symmetric spaces to the category of Lie triple systems.

A Banach space can be considered as a symmetric space with natural multiplication $x\cdot y :=2x-y$. From this perspective, a smooth curve $\alpha\colon \RR\rightarrow M$ is called a \emph{one-parameter subspace of $M$} if $\alpha$ is a morphism of symmetric spaces. For each $v\in \Lts(M,b)$, there is a unique one-parameter subspace $\alpha_v$ with $\alpha_v^\prime(0)=v$. The map
$\Exp_{(M,b)}\colon \Lts(M,b)\rightarrow M$, $v\mapsto \alpha_v(1)$
is called the \emph{exponential map of $(M,b)$}. It is a smooth map with $T_0\Exp_{(M,b)}=\id_{\Lts(M,b)}$, so that it is a local diffeomorphism at $0$ and hence admits restrictions that are charts at $b$ (called \emph{normal charts}).
A morphism $f\colon (M_1,b_1)\rightarrow (M_2,b_2)$ of pointed symmetric spaces intertwines the exponential maps in the sense that
$f\circ \Exp_{(M_1,b_1)} = \Exp_{(M_2,b_2)} \circ \Lts(f)$. For details concerning the exponential map, see \cite{Klo09b}, whose approach is based on affine connections.\footnote{In \cite{Klo09b}, one-parameter subspaces are not dealt with, but it is easy to check that they coincide with geodesics. Cf.\ \cite[p.~87]{Loo69} for the finite-dimensional case.}

Given a one-parameter subspace $\alpha$, we call the automorphisms $\tau_{\alpha,s}:=\mu_{\alpha(\frac{1}{2}s)}\circ\mu_{\alpha(0)}$, $s\in\RR$, of $M$ \emph{translations along $\alpha$}. They satisfy $\tau_{\alpha,s}(\alpha(t))=\alpha(t+s)$ for all $t\in\RR$.
If $M$ is connected, then any two points can be joined by a sequence of one-parameter subspaces, since we have normal charts. Therefore, in view of the identities $\alpha(1)=(\tau_{\alpha,\frac{1}{n}})^n(\alpha(0))$ for all $n\in\NN$, it is easy to see that a connected $M$ is generated by each subset $U\subseteq M$ with non-empty interior. As a consequence, the \emph{basic connected component} $M_0$ of $(M,b)$ is generated by the image of the exponential map $\Exp_{(M,b)}$.

The automorphism group $\Aut(M)$ of a reflection space $M$ (resp., of a symmetric space) has two important (normal) subgroups:
The set of all symmetries $\mu_x$, $x\in M$, generates a subgroup which is denoted by $\Inn(M)$ and is called the \emph{group of inner automorphisms}. The subgroup $G(M)$ generated by all products $\mu_x\mu_y$, $x,y\in M$, is called the \emph{group of displacements} (cf.\ \cite[p.~64]{Loo69}).
For a connected symmetric space $M$, these groups of automorphisms act transitively on $M$, since there are translations along one-parameter subspaces.

Given a homomorphism $f\colon M_1\rightarrow M_2$, then for all $g_1\in\Inn (M_1)$, there exists a (not necessarily unique) $g_2\in\Inn(M_2)$ with
\begin{equation} \label{eqn:imageOfInnAut}
	f\circ g_1 \ =\ g_2 \circ f,
\end{equation}
because for a decomposition $g_1= \mu_{x_1}\mu_{x_2}\cdots\mu_{x_n}$, we can put $g_2:=\mu_{f(x_1)}\mu_{f(x_2)}\cdots\mu_{f(x_n)}$. As a consequence, a homomorphism $f\colon M_1\rightarrow M_2$ of reflection spaces that is locally smooth around some $b\in M_1$ is automatically smooth (and hence a morphism of symmetric spaces) if $M_1$ is $\Inn(M_1)$-transitive. Thus, given pointed symmetric spaces $(M_1,b_1)$ and $(M_2,b_2)$ with Lie triple systems $\mathfrak{m}_1$ and $\mathfrak{m}_2$, respectively, then a homomorphism $f\colon (M_1,b_1)\rightarrow (M_2,b_2)$ of pointed reflection spaces that satisfies
$f\circ \Exp_{(M_1,b_1)} = \Exp_{(M_2,b_2)} \circ A$
for some continuous linear map $A\colon \mathfrak{m}_1\rightarrow \mathfrak{m}_2$  is a morphism of pointed symmetric spaces with $\Lts(f)=A$.

%

%
%
\subsection{Symmetric Lie Algebras and Lie Groups}\label{sec:symLieAlgAndLieGrp}
A \emph{symmetric Lie algebra} is a Banach--Lie algebra $\mathfrak{g}$ endowed with an involutive automorphism $\theta$ of $\mathfrak{g}$, i.e., $\theta^2=\id_\mathfrak{g}$. Given two symmetric Lie algebras $(\mathfrak{g}_1,\theta_1)$ and $(\mathfrak{g}_2,\theta_2)$, a continuous Lie algebra homomorphism $A\colon \mathfrak{g}_1 \rightarrow \mathfrak{g}_2$ is called a \emph{morphism of symmetric Lie algebras} if it satisfies $A\circ\theta_1=\theta_2\circ A$.
The kernel of a morphism $A\colon (\mathfrak{g}_1,\theta_1)\rightarrow (\mathfrak{g}_2,\theta_2)$ satisfies $\theta_1(\ker(A))\subseteq\ker(A)$, so that $(\ker(A),\theta_1|_{\ker(A)}^{\ker(A)})$ is a symmetric Lie algebra.
A symmetric Lie algebra $(\mathfrak{g},\theta)$ decomposes as the direct sum of its $(\pm 1)$-eigenspaces denoted by $\mathfrak{g}=\mathfrak{g}_+\oplus \mathfrak{g}_-$. The $(+1)$-eigenspace is a subalgebra of $\mathfrak{g}$. The $(-1)$-eigenspace $\mathfrak{g}_-$ becomes a Lie triple system by defining $[x,y,z]:=[[x,y],z]$ (cf.\ \cite[Prop.~5.9]{Klo09b}).
The adjoint representation $\ad\colon \mathfrak{g}\rightarrow \gl(\mathfrak{g})$, $x\rightarrow [x,\cdot]$ induces a representation $\mathfrak{g}_+\rightarrow \gl(\mathfrak{g}_-)$, $x \mapsto [x,\cdot]|_{\mathfrak{g}_-}$.
Assigning to each symmetric Lie algebra $(\mathfrak{g},\theta)$ the Lie triple system $\mathfrak{g}_-$ and to each morphism of symmetric Lie algebras its restriction to the Lie triple systems, we obtain a covariant functor $\calLts$.

%

A \emph{symmetric Lie group} is a Banach--Lie group $G$ together with an involutive automorphism $\sigma$ of $G$. A \emph{morphism between symmetric Lie groups $(G_1,\sigma_1)$ and $(G_2,\sigma_2)$} is a smooth group homomorphism $f\colon G_1\rightarrow G_2$ such that $f\circ\sigma_1=\sigma_2\circ f$.
The Lie functor $L$ assigns to each $(G,\sigma)$ the \emph{symmetric Lie algebra} $L(G,\sigma):=(L(G),L(\sigma))$ and to each morphism $f$ the morphism $L(f)$ of symmetric Lie algebras.
The kernel $\ker(f)\unlhd G_1$ of a morphism $f$ is a Lie subgroup (cf.\ \cite[Th.~II.2]{GN03}) that satisfies $\sigma_1(\ker(f))\subseteq\ker(f)$, so that $(\ker(f),\sigma_1|_{\ker(f)}^{\ker(f)})$ is a symmetric Lie group. Its Lie algebra is $(\ker(L(f)),L(\sigma_1)|_{\ker(L(f))}^{\ker(L(f))})$.

Given a symmetric Lie group $(G,\sigma)$ with symmetric Lie algebra $(\mathfrak{g},\theta)$, the subgroup $G^\sigma:=\{g\in G\colon \sigma(g)=g\}$ of $\sigma$-fixed points is a split Lie subgroup with Lie algebra $\mathfrak{g}_+$ (cf.\ \cite[Ex.~3.9]{Nee02Cartan}). Open subgroups of $G^\sigma$ are given by subgroups $K\leq G^\sigma$ satisfying $(G^\sigma)_0\subseteq K \subseteq G^\sigma$, where $(G^\sigma)_0$ denotes the identity component. For such a subgroup $K$, we call $((G,\sigma),K)$ a \emph{symmetric pair} and shall rather write $(G,\sigma,K)$.\footnote{Note that other authors do not include the involution $\sigma$ in their definition, but require its existence (cf.\ \cite{Hel01}).}
The quotient space $\Sym(G,\sigma,K):=G/K$ carries the structure of a pointed symmetric space with multiplication
$$gK \cdot hK:=g\sigma(g)^{-1}\sigma(h)K$$
and base point $K$ such that the quotient map $q\colon G\rightarrow G/K$ is a submersion.
Note that $q$ is a principal bundle with structure group $K$ 
that acts on $G$ by right translations (cf.\ \cite[III.1.5--6]{Bou89LieGroups}).
When we consider the underlying symmetric space of $G/K$ that is not pointed, then we shall frequently write $\calU(G/K)$ to prevent confusion.
The Lie triple system $\Lts(G/K)$ can be identified with $\mathfrak{g}_-$ via the isomorphism $(T_{\eins}q)|_{\mathfrak{g}_-}$. Then the exponential map of $G/K$ is given by
\begin{equation}\label{eqn:Exp=q exp}
	\Exp_{G/K}:=q\circ\exp_G|_{\mathfrak{g}_-}\colon \mathfrak{g}_- \rightarrow G/K,
\end{equation}
where $\exp_G$ denotes the exponential map of the Lie group $G$. For further details, cf.\ \cite[Ex.~3.9]{Nee02Cartan}.

Considering the transitive smooth action
\begin{equation} \label{eqn:tau}
	\tau\colon G\times G/K\rightarrow G/K,\ (g,hK)\mapsto ghK,
\end{equation}
it is easy to check that the corresponding diffeomorphisms $\tau_g\colon G/K\rightarrow G/K$, $hK\mapsto ghK$ are automorphisms of the symmetric space $\calU(G/K)$ and that
\begin{equation} \label{eqn:tau_g^2}
	\tau_g^2 \ =\  \mu_{gK}\circ \mu_K \quad\mbox{for all $g\in G$ with $\sigma(g)=g^{-1}$}.
\end{equation}


A \emph{morphism between symmetric pairs $(G_1,\sigma_1,K_1)$ and $(G_2,\sigma_2,K_2)$} is a morphism \linebreak $f\colon (G_1,\sigma_1)\rightarrow (G_2,\sigma_2)$ of symmetric Lie groups satisfying $f(K_1)\subseteq K_2$. Note that $f(G_1^{\sigma_1})\subseteq G_2^{\sigma_2}$ and $f((G_1^{\sigma_1})_0)\subseteq (G_2^{\sigma_2})_0$ are always satisfied and that we moreover have $f((G_1^{\sigma_1})_0)= (G_2^{\sigma_2})_0$ if $L(f)((\mathfrak{g}_1)_+)=(\mathfrak{g}_2)_+$.\footnote{Indeed, we have $f((G_1^{\sigma_1})_0) \ =\ \langle f(\exp_{G_1}((\mathfrak{g}_1)_+))\rangle \ =\ \langle\exp_{G_2}((\mathfrak{g}_2)_+) \rangle \ =\ (G_2^{\sigma_2})_0.$}\label{page:L(f)(g1_+)=g2_+}
Every morphism $f$ induces a unique map
$$\Sym(f):=f_\ast\colon G_1/K_1 \rightarrow G_2/K_2$$
with $\Sym(f)\circ q_1 = q_2\circ f$ that is automatically a morphism of pointed symmetric spaces.
The assignment $\Sym$ is a covariant functor from the category of symmetric pairs to the category of pointed symmetric spaces. Denoting by $F$ the forgetful functor from the category of symmetric pairs to the category of symmetric Lie groups, we have
\begin{equation} \label{eqn:LtsSym=calLtsL}
	\Lts\circ\Sym \ =\ \calLts\circ L \circ F
\end{equation}
(where we read $L$ as the Lie functor applied to symmetric Lie groups). Indeed, the map $\Lts(\Sym(f))\colon (\mathfrak{g}_1)_- \rightarrow (\mathfrak{g}_2)_-$ is given by the restriction $L(f)|_{(\mathfrak{g}_1)_-}^{(\mathfrak{g}_2)_-}$ of $L(f)$, because  we have
\begin{eqnarray*}
	\Exp_{G_2/K_2}\circ L(f)|_{(\mathfrak{g}_1)_-}^{(\mathfrak{g}_2)_-} &=& q_2\circ\exp_{G_2} \circ L(f)|_{(\mathfrak{g}_1)_-}^{(\mathfrak{g}_2)_-} \ =\ q_2\circ f\circ \exp_{G_1}|_{(\mathfrak{g}_1)_-} \\
	&=& \Sym(f)\circ q_1 \circ \exp_{G_1}|_{(\mathfrak{g}_1)_-} \ =\ \Sym(f)\circ \Exp_{G_1/K_1}.
\end{eqnarray*}

Given a morphism $f\colon (G_1,\sigma_1,K_1)\rightarrow (G_2,\sigma_2,K_2)$ of symmetric pairs, then for all $g\in G_1$, we have
\begin{equation} \label{eqn:SymTau}
	\Sym(f)\circ \tau_g^{G_1} \ =\ \tau_{f(g)}^{G_2}\circ \Sym(f),
\end{equation}
where $\tau^{G_i}$ (with $i=1,2$) denotes the natural action of $G_i$ on $G_i/K_i$ (cf.\ (\ref{eqn:tau})).
\begin{lemma}
\label{lem:symToInjevtive}
	Let $f\colon (G_1,\sigma_1)\rightarrow (G_2,\sigma_2)$ be an injective morphism of a symmetric Lie group $(G_1,\sigma_1)$ to the underlying symmetric Lie group of a symmetric pair $(G_2,\sigma_2,K_2)$. Then $(G_1,\sigma_1,K_1)$ with $K_1:=f^{-1}(K_2)$ is a symmetric pair turning $f$ into a morphism of symmetric pairs such that the morphism $\Sym(f)\colon G_1/K_1\rightarrow G_2/K_2$ is also injective.	(Note that $G_1^{\sigma_1}=f^{-1}(G_2^{\sigma_2})$.)
\end{lemma}
\begin{proof}
	It is easy to see that for all $x\in G_1$, the condition $x\in G_1^{\sigma_1}$ holds if and only if $f(x) \in G_2^{\sigma_2}$. Hence, we have $G_1^{\sigma_1}=f^{-1}(G_2^{\sigma_2})$, so that $K_1:=f^{-1}(K_2)$ is an open subset of $G_1^{\sigma_1}$, the subgroup $K_2$ being open in $G_2^{\sigma_2}$. It is clear that $f$ becomes a morphism of symmetric pairs. To see the injectivity of $\Sym(f)$, we take any $gK_1, hK_1\in G_1/K_1$ satisfying $\Sym(f)(gK_1)=\Sym(f)(hK_1)$ and shall show $gK_1=hK_1$. Indeed, by assumption, we have $f(g)K_2=f(h)K_2$, entailing $f(h^{-1}g)\in K_2$, so that $h^{-1}g\in K_1$. Hence, we obtain $gK_1=hK_1$.
\end{proof}
\subsection{The Automorphism Group of a Connected Symmetric Space}
\label{sec:AutOfConnectedSymSpace}
Let $M$ be a connected symmetric space. The Lie algebra $\Der(M)$ of derivations
can be uniquely turned into a Banach--Lie algebra such that for each frame $p$ in the frame bundle $\Fr(M)$ over $M$, the map
$$\Der(M)\rightarrow T_p(\Fr(M)),\ \xi\mapsto \left.\frac{d}{dt}\right|_{t=0}\Fr(\flow^{\xi}_t)(p)$$
is an embedding of Banach spaces, where $\flow^{\xi}_t$ is the time-$t$-flow of $\xi$ and $\Fr(\flow^{\xi}_t)$ is its induced automorphism of the frame bundle (cf.\ \cite[Cor.~5.18]{Klo09b}).
Given a base point $b\in M$, the involutive automorphism $(\mu_b)_\ast$ of the Lie algebra $\Der(M)$ is actually continuous and hence an automorphism of the Banach--Lie algebra $\Der(M)$, so that $(\Der(M),(\mu_b)_\ast)$ is a symmetric (Banach--)Lie algebra  and the evaluation map $\ev_b\colon\Der(M)_-\rightarrow \Lts(M,b)$ is an isomorphism of Lie triple systems (cf.\ \cite[Prop.~5.23]{Klo09b}).

The automorphism group $\Aut(M)$ can be turned into a Banach--Lie group such that
$$\exp\colon \Der(M)\rightarrow \Aut(M),\ \xi \mapsto \flow^{-\xi}_1$$
is its exponential map. The natural map $\tau\colon \Aut(M)\times M \rightarrow M$ is a transitive smooth action.
Together with the conjugation map $c_{\mu_b}\colon \Aut(M)\rightarrow \Aut(M)$, $g\mapsto\mu_b\circ g \circ \mu_b$, the automorphism group $\Aut(M)$ becomes a symmetric Lie group with Lie algebra $L(\Aut(M),c_{\mu_b})=(\Der(M),(\mu_b)_\ast)$. The stabilizer subgroup $\Aut(M)_b\leq \Aut(M)$ leads to the symmetric pair $(\Aut(M),c_{\mu_b},\Aut(M)_b)$. The induced symmetric space $\Aut(M)/\Aut(M)_b$ is isomorphic to $M$ via the isomorphism $\Phi\colon \Aut(M)/\Aut(M)_b\rightarrow M$ given by $\Phi(g\Aut(M)_b):=g(b)$. We refer to this fact as the \emph{homogeneity of connected symmetric spaces}. For further details, see \cite[Sec.~5.5]{Klo09b}.

%
%
\section{Reflection Subspaces and Quotients of Symmetric Spaces} \label{sec:SubspacesAndQuotients}
In this section, we deal with reflection subspaces and quotients of symmetric spaces. Taking advantage of the homogeneity of connected symmetric spaces, we shall partially translate knowledge from Lie groups to symmetric spaces. A useful source for writing certain parts of this section was a course held by H.~Gl\"ockner, where analogous concepts for Lie groups were dealt with.\footnote{\emph{Infinite-dimensional Lie groups}, a course held by H.~Gl\"ockner at the Technical University of Darmstadt in the Summer Semester 2005}
%
%
\subsection{Closed Reflection Subspaces} \label{sec:closedSubspaces}
In this subsection, we shall assign Lie triple systems to closed reflection subspaces of pointed symmetric spaces. For this, we require formulas comparable to the well-known Trotter product formula (cf.\ \cite{Tro59}) and the commutator formula in the context of Lie groups (cf.\ \cite[Prop.~8 in III.6.4]{Bou89LieGroups}).
Given a Banach-Lie group $G$ and any $x$ and $y$ in its Lie algebra $L(G)$, we know that
\begin{equation} \label{eqn:TrotterProductFormula}
	\exp(x+y) \ =\ \lim_{k\rightarrow\infty}\Big(\exp\big(\frac{x}{k}\big)\exp\big(\frac{y}{k}\big)\Big)^k
\end{equation}
and
\begin{equation}\label{eqn:commutatorFormula}
	\exp([x,y]) \ =\ \lim_{k\rightarrow\infty}\Big(\exp\big(\frac{x}{k}\big)\exp\big(\frac{y}{k}\big)\exp\big(-\frac{x}{k}\big)\exp\big(-\frac{y}{k}\big)\Big)^{k^2}.
\end{equation}
The following lemma will aid us to translate these known formulas into the context of symmetric spaces.
\begin{lemma} \label{lem:translatorForTrotter}
	Let $(G,\sigma,K)$ be a symmetric pair with quotient map $q\colon G\rightarrow G/K$ and let $\mathfrak{g}=\mathfrak{g}_+ \oplus \mathfrak{g}_-$ be the Lie algebra of its underlying symmetric Lie group $(G,\sigma)$. We denote the exponential maps of $G$ by $\exp$ and of $G/K$ by $\Exp$.
	Given any $n\in\NN$, we have
	$$q\big(\exp(x_n)\exp(y_n)\cdots\exp(x_1)\exp(y_1)\big) \ =\ \big(\mu_{\Exp(\frac{x_n}{2})}\mu_{\Exp(-\frac{y_n}{2})} \cdots \mu_{\Exp(\frac{x_1}{2})}\mu_{\Exp(-\frac{y_1}{2})}\big)(K) $$
	for all $x_1,y_1,\ldots, x_n, y_n\in \mathfrak{g}_-$,
	where $\mu$ denotes the multiplication map on $G/K$.
\end{lemma}
\begin{proof}
	For all $x\in\mathfrak{g}_-$, we have $\sigma(\exp(x))=\exp(-x)=(\exp(x))^{-1}$ because of $\sigma\circ\exp=\exp\circ L(\sigma)$ and $L(\sigma)|_{\mathfrak{g}_-}=-\id_{\mathfrak{g}_-}$. It follows that $\tau^2_{\exp(x)}=\mu_{\exp(x)K}\circ\mu_{K} = \mu_{\Exp(x)}\circ\mu_{K}$ for all $x\in\mathfrak{g}_-$ (cf.\ (\ref{eqn:tau_g^2})), so that we have
	$$\tau_{\exp(x)} \ =\ \tau^2_{\exp(\frac{x}{2})} \ =\ \mu_{\Exp(\frac{x}{2})}\circ\mu_{K} \quad\mbox{for all $x\in\mathfrak{g}_-$}$$
	and hence
	$$\tau_{\exp(x)\exp(y)} \ =\ \tau_{\exp(x)}\circ\tau^{-1}_{\exp(-y)} \ =\ \mu_{\Exp(\frac{x}{2})} \circ \mu_{\Exp(-\frac{y}{2})} \quad\mbox{for all $x,y\in\mathfrak{g}_-$.}$$
	Therefore we obtain
	\begin{align*}
		q\big(\exp(x_n)\exp(y_n)\cdots\exp(x_1)\exp(y_1)\big) \ &=\ \big(\tau_{\exp(x_n)\exp(y_n)}\circ\cdots\circ\tau_{\exp(x_1)\exp(y_1)}\big)(K) \\
		&=\ \big(\mu_{\Exp(\frac{x_n}{2})}\mu_{\Exp(-\frac{y_n}{2})} \cdots \mu_{\Exp(\frac{x_1}{2})}\mu_{\Exp(-\frac{y_1}{2})}\big)(K).
		\qedhere
	\end{align*}
\end{proof}
\begin{proposition}\label{prop:TrotterForSymSpace}
	Let $(M,b)$ be a pointed symmetric space. For all $x,y,z\in\Lts(M,b)$, we have the formulas
	$$\Exp_{(M,b)}(x+y) \ =\ \lim_{k\rightarrow\infty}\big(\mu_{\Exp_{(M,b)}(\frac{x}{2k})} \mu_{\Exp_{(M,b)}(-\frac{y}{2k})}\big)^k(b)$$
	and
	$$\Exp_{(M,b)}([x,y,z]) \ =\ \lim_{k\rightarrow\infty}\,\lim_{l\rightarrow\infty} \big( g_{(k,l)}\mu_{\Exp_{(M,b)}(\frac{z}{2k})}h_{(k,l)}\mu_{\Exp_{(M,b)}(\frac{z}{2k})}\big)^{k^2}(b)$$
	with
	 $$g_{(k,l)}:=  \big(\mu_{\Exp_{(M,b)}(\frac{x}{2l\sqrt{k}})}\mu_{\Exp_{(M,b)}(-\frac{y}{2l\sqrt{k}})}\mu_{\Exp_{(M,b)}(-\frac{x}{2l\sqrt{k}})}\mu_{\Exp_{(M,b)}(\frac{y}{2l\sqrt{k}})}\big)^{l^2}$$
	and
	$$h_{(k,l)}:=  \big(\mu_{\Exp_{(M,b)}(\frac{x}{2l\sqrt{k}})}\mu_{\Exp_{(M,b)}(\frac{y}{2l\sqrt{k}})}\mu_{\Exp_{(M,b)}(-\frac{x}{2l\sqrt{k}})}\mu_{\Exp_{(M,b)}(-\frac{y}{2l\sqrt{k}})}\big)^{l^2}.$$
\end{proposition}
\begin{proof}
	W.l.o.g., we assume $M$ to be connected, so that we can identify it with the homogeneous space $\Aut(M)/\Aut(M)_b$ (in the following abbreviated by $G/G_b$). Applying Lemma~\ref{lem:translatorForTrotter} to the quotient map $q\colon G\rightarrow G/G_b$, we obtain
	\begin{eqnarray*}
		\Exp_{G/G_b}(x+y) &=& q(\exp_{G}(x+y)) \ \stackrel{\mbox{\scriptsize (\ref{eqn:TrotterProductFormula})}}{=}\ \lim_{k\rightarrow\infty} q\Big(\big(\exp_{G}\big(\frac{x}{k}\big)\exp_{G}\big(\frac{y}{k}\big)\big)^k\Big)\\
		 &=& \lim_{k\rightarrow\infty}\big(\mu_{\Exp_{G/G_b}(\frac{x}{2k})}\mu_{\Exp_{G/G_b}(-\frac{y}{2k})}\big)^k(G_b).
	\end{eqnarray*}
	Further, we have	
	\begin{eqnarray*}
		\Exp_{G/G_b}([x,y,z]) &=& q\big(\exp_{G}\big([[x,y],z]\big)\big) \\ &\stackrel{\mbox{\scriptsize (\ref{eqn:commutatorFormula})}}{=}& \lim_{k\rightarrow\infty} q\Big(\big(\exp_{G}\big(\frac{[x,y]}{k}\big)\exp_{G}\big(\frac{z}{k}\big)\exp_{G}\big(-\frac{[x,y]}{k}\big)\exp_{G}\big(-\frac{z}{k}\big)\big)^{k^2}\Big)
	\end{eqnarray*}
	with
	$$\exp_G\big(\frac{[x,y]}{k}\big) \ \stackrel{\mbox{\scriptsize (\ref{eqn:commutatorFormula})}}{=}\  \lim_{l\rightarrow\infty}\Big(\exp_{G}\big(\frac{x}{l\sqrt{k}}\big)\exp_{G}\big(\frac{y}{l\sqrt{k}}\big)\exp_{G}\big(-\frac{x}{l\sqrt{k}}\big)\exp_{G}\big(-\frac{y}{l\sqrt{k}}\big)\Big)^{l^2}$$
	and
	$$\exp_G\big(-\frac{[x,y]}{k}\big) \ \stackrel{\mbox{\scriptsize (\ref{eqn:commutatorFormula})}}{=}\  \lim_{l\rightarrow\infty}\Big(\exp_{G}\big(-\frac{x}{l\sqrt{k}}\big)\exp_{G}\big(\frac{y}{l\sqrt{k}}\big)\exp_{G}\big(\frac{x}{l\sqrt{k}}\big)\exp_{G}\big(-\frac{y}{l\sqrt{k}}\big)\Big)^{l^2},$$
	so that applying again Lemma~\ref{lem:translatorForTrotter} leads to the assertion.
\end{proof}
\begin{proposition} \label{prop:LtsOfClosedReflectionSubspace}
	Let $(M,b)$ be a pointed symmetric space and $(N,b)\leq (M,b)$ be a closed reflection subspace. Then
	$\mathfrak{n}:=\{x\in\Lts(M,b)\colon \Exp_{(M,b)}(\RR x)\subseteq N\}$
	is a closed triple subsystem of $\Lts(M,b)$.
\end{proposition}
\begin{definition}\label{def:LtsOfClosedSubreflectionSpace}
	We call $\mathfrak{n}\leq\Lts(M,b)$ the \emph{Lie triple system of the closed reflection subspace $(N,b)\leq (M,b)$}.
\end{definition}
\begin{proof}
	Cf.\ \cite[p.~126]{Loo69} for the finite-dimensional case. By the continuity of the exponential map, it is clear that $\mathfrak{n}$ is a closed subset of $\Lts(M,b)$. Further, it its trivial that $\RR\mathfrak{n}\subseteq\mathfrak{n}$. Given any $x,y,z\in\mathfrak{n}$, Proposition~\ref{prop:TrotterForSymSpace} shows that $\Exp_{(M,b)}(tx+ty)$ and $\Exp_{(M,b)}([tx,ty,tz])$ lie in $N$ for all $t\in\RR$, since $N$ is a closed reflection subspace of $M$. Thus, we have $x+y\in\mathfrak{n}$ and $[x,y,z]\in\mathfrak{n}$, entailing that $\mathfrak{n}$ is a triple subsystem of $\Lts(M,b)$.
\end{proof}
\begin{lemma}\label{lem:subsetN_1=f^-1(N_2)}
	Let $f\colon (M_1,b_1)\rightarrow (M_2,b_2)$ be a morphism of pointed symmetric spaces and $N_2\subseteq M_2$ a subset. We consider the sets $N_1:=f^{-1}(N_2)$ and 
	$$\mathfrak{n}_i:=\{x\in\Lts(M_i,b_i)\colon \Exp_{(M_i,b_i)}(\RR x)\subseteq N_i\}$$
	for $i=1,2$. Then we have $\mathfrak{n}_1=\Lts(f)^{-1}(\mathfrak{n}_2)$.
\end{lemma}
\begin{proof}
	Given any $x\in \Lts(M_1,b_1)$, we have $x\in\mathfrak{n}_1$ if and only if $f(\Exp_{(M_1,b_1)}(\RR x))\subseteq N_2$, which is equivalent to $\Exp_{(M_2,b_2)}(\RR \Lts(f)(x))\subseteq N_2$, i.e., to $\Lts(f)(x)\in \mathfrak{n}_2$. Hence, it follows that $\mathfrak{n}_1=\Lts(f)^{-1}(\mathfrak{n}_2)$.
\end{proof}
\begin{proposition}\label{prop:preimageOfClosedReflectionSubspace}
	Let $f\colon (M_1,b_1)\rightarrow (M_2,b_2)$ be a morphism of pointed symmetric spaces and $(N_2,b_2)\leq (M_2,b_2)$ a closed reflection subspace with Lie triple system $\mathfrak{n}_2\leq \Lts(M_2,b_2)$.\linebreak Then $N_1:=f^{-1}(N_2)$ is a closed reflection subspace of $(M_1,b_1)$ with Lie triple system\linebreak $\mathfrak{n}_1=\Lts(f)^{-1}(\mathfrak{n}_2)$.
\end{proposition}
\begin{proof}
	The assertion follows immediately by Definition~\ref{def:LtsOfClosedSubreflectionSpace} and Lemma~\ref{lem:subsetN_1=f^-1(N_2)}.
\end{proof}
%
%
%
%
\subsection{Integral Subspaces} \label{sec:integralSubspaces}
An \emph{integral subspace} of a symmetric space $M$ is a reflection subspace $N\leq M$ endowed with a symmetric space structure such that the inclusion map $\iota\colon N\rightarrow M$ is smooth and for each $b\in N$, the induced morphism $\Lts_b(\iota)\colon \Lts(N,b)\rightarrow \Lts(M,b)$ of Lie triple systems is a (closed) topological embedding.
In the light of (\ref{eqn:imageOfInnAut}), we know for $\Inn(N)$-transitive (e.g.\ connected) $N$ that, for each $b_1,b_2\in N$, the map $\Lts_{b_1}(\iota)$ is a topological embedding if and only if $\Lts_{b_2}(\iota)$ is one.

We shall frequently identify $\Lts(N,b)$ with its image $\mathfrak{n}\subseteq \Lts(M,b)$ under $\Lts_b(\iota)$. Thus, the exponential map $\Exp_{(N,b)}$ of $(N,b)$ is the restriction $\Exp_{(M,b)}|_\mathfrak{n}$ of the exponential map $\Exp_{(M,b)}$ of $(M,b)$.
The basic connected component $N_0$ of $N$ is given by $N_0=\langle\Exp_{(M,b)}(\mathfrak{n})\rangle$ (cf.\ Section~\ref{sec:LtsAndSymSpace}).

\begin{example} \label{ex:integralSubspaceOfHomogeneousSpace}
	Let $(G,\sigma_G,K_G)$ be a symmetric pair and $\iota\colon (H,\sigma_H)\rightarrow (G,\sigma_G)$ be an injective morphism of symmetric Lie groups such that $\iota\colon H\rightarrow G$ is an integral subgroup.
	With $K_H:=\iota^{-1}(K_G)$, the map $\iota\colon (H,\sigma_H,K_H)\rightarrow (G,\sigma_G,K_G)$ is a morphism of symmetric pairs and $\Sym(\iota)\colon H/K_H\rightarrow G/K_G$ is an injective morphism of pointed symmetric spaces (cf.\ Lemma~\ref{lem:symToInjevtive}). The induced map $\Lts(\Sym(\iota))\colon \Lts(H/K_H)\rightarrow \Lts(G/K_G)$ is given by $\calLts(L(\iota))$ (cf.\ (\ref{eqn:LtsSym=calLtsL})), i.e., by the topological embedding 
	$\mathfrak{h}_- \hookrightarrow \mathfrak{g}_-$,
	where $\mathfrak{g}$ and $\mathfrak{h}$ denote the Lie algebras of $G$ and $H$, respectively.
	
	Therefore $\Sym(\iota)$ is an integral subspace if $H/K_H$ is $\Inn(\calU(H/K_H))$-transitive, but this additional assumption is not necessary.
	Indeed, to see that, for each $h\in H$, the map $\Lts_{hK_H}(\Sym(\iota))$ is a topological embedding, it suffices to note that
	$$\Lts_{hK_H}(\Sym(\iota))\circ\Lts_{K_H}(\tau_h^H) \ =\ \Lts_{K_G}(\tau_{\iota(h)}^G)\circ\Lts_{K_H}(\Sym(\iota))$$
	(cf.\ (\ref{eqn:SymTau})), where $\tau_h^H$ and $\tau_{\iota(h)}^G$ are automorphisms of symmetric spaces.
\end{example}
\begin{lemma}\label{lem:uniquenessOfIntegralSubspaces}
	Let $(M,b)$ be a pointed symmetric space, $(N,b)\leq (M,b)$ an $\Inn(N)$-transitive reflection subspace and $\mathfrak{n}$ a closed triple subsystem of $\Lts(M,b)$ such that $\Exp_{(M,b)}(\mathfrak{n})\subseteq N$. Let $\iota\colon (N,b)\rightarrow (M,b)$ be the inclusion map. Then there is at most one integral subspace structure on $N$ which makes $\mathfrak{n}$ the Lie triple system of $N$, i.e., $\im (\Lts(\iota))=\mathfrak{n}$.
\end{lemma}
\begin{proof}
	Let $\calA_1$ and $\calA_2$ be smooth atlases on $N$, each of them making it an integral subspace with Lie triple system $\mathfrak{n}$. Since we have $\id_N\circ \Exp_{(M,b)}|_{\mathfrak{n}} = \Exp_{(M,b)}|_{\mathfrak{n}}\circ \id_{\mathfrak{n}}$, the map $\id_N\colon (N,\calA_1)\rightarrow (N,\calA_2)$ is smooth (cf.\ Section~\ref{sec:LtsAndSymSpace}). The inverse map $\id_N\colon (N,\calA_2)\rightarrow (N,\calA_1)$ being smooth, too, the map $\id_N$ is an isomorphism of integral subspaces.
\end{proof}
\begin{remark} \label{rem:uniquenessOfIntegralSubspaceH/K_H}
	In Example~\ref{ex:integralSubspaceOfHomogeneousSpace}, the integral subspace $H/K_H$ of $G/K_G$ is not assumed to be $\Inn(\calU(H/K_H))$-transitive, but, nevertheless, there is no other integral subspace structure on (the reflection subspace) $H/K_H$ with Lie triple system $\mathfrak{h}_-\hookrightarrow \mathfrak{g}_-$ if we require, in addition, that the natural action $\tau^H\colon H \times H/K_H \rightarrow H/K_H$ (cf.\ (\ref{eqn:tau})) is smooth.
	(Actually, it suffices to require that it induces smooth maps $\tau^H_h$ for all $h\in H$.)
	Indeed, with $(M,b):=G/K_G$ and $(N,b):=H/K_H$, the equation in the proof of Lemma~\ref{lem:uniquenessOfIntegralSubspaces} still shows that $\id_N$ is at least locally smooth around $b$. Since we have $\id_N\circ \tau_h^{H} = \tau_{h}^{H}\circ \id_N$ for all $h\in H$, the map $\id_N$ is smooth around every point in $N$ and hence smooth.
\end{remark}
Before giving a proposition about integrating a closed triple subsystem of the Lie triple system of a pointed symmetric space to an integral subspace, we consider the following lemmas:
\begin{lemma}\label{lem:integralSubspaceOfSymLieGroup}
	Let $(G,\sigma)$ be a symmetric Lie group and let $\mathfrak{h}\leq L(G)$ be a $L(\sigma)$-invariant closed subalgebra of the Lie algebra $L(G)$. Then the connected integral subgroup $H:=\langle\exp(\mathfrak{h})\rangle$ of $G$ with Lie algebra $\mathfrak{h}$ is $\sigma$-invariant.
	Further, $(H,\sigma_H)$ with $\sigma_H:=\sigma|_H^H$ is a symmetric Lie group with symmetric Lie algebra $(\mathfrak{h},L(\sigma)|_\mathfrak{h}^\mathfrak{h})$.
\end{lemma}
\begin{proof}
	The $\sigma$-invariance of $H$ can be checked by a standard argument. From $\exp|_{\mathfrak{h}}\circ L(\sigma)|_{\mathfrak{h}}^{\mathfrak{h}} = \sigma_{H} \circ \exp|_{\mathfrak{h}}$, we deduce that the map $\sigma_H$ is smooth on an open identity neighborhood and hence globally smooth being a group homomorphism. It follows that $L(\sigma_H)=L(\sigma)|_\mathfrak{h}^\mathfrak{h}$.
\end{proof}
\begin{lemma}\label{lem:[n,n]+n}
	Let $(\mathfrak{g},\theta)$ be a symmetric Lie algebra and let $\mathfrak{n}$ be a closed triple subsystem\linebreak of $\mathfrak{g}_-$. Then the closed subspace $\mathfrak{h}:=\overline{[\mathfrak{n},\mathfrak{n}]}\oplus \mathfrak{n}\leq \mathfrak{g}_+\oplus\mathfrak{g}_-$ of $\mathfrak{g}$ is a $\theta$-invariant subalgebra\linebreak of $\mathfrak{g}$, hence a symmetric Lie algebra with $\mathfrak{h}_+=\overline{[\mathfrak{n},\mathfrak{n}]}$ and $\mathfrak{h}_-=\mathfrak{n}$.
\end{lemma}
\begin{proof}
	Cf.\ \cite[p.~122]{Loo69} or \cite[Ch.~IV, \S~7]{Hel01} for the finite-dimensional case. The proof carries over to the Banach case.
\end{proof}
\begin{proposition} \label{prop:integralSubreflectionSpace}
	Let $(M,b)$ be a pointed symmetric space and $\mathfrak{n}$ a closed triple subsystem of $\Lts(M,b)$. Then $N:=\langle \Exp_{(M,b)}(\mathfrak{n})\rangle \leq M$ can be uniquely made an integral subspace of $M$ with $\Lts(N,b)=\mathfrak{n}$. Note that $N$ is connected.
\end{proposition}
\begin{proof}
	Since $N$ is a reflection subspace of the basic connected component $M_0$, we can, w.l.o.g., assume $M$ to be connected and identify it with the homogeneous space $\Aut(M)/\Aut(M)_b$, in the following abbreviated by $G/G_b$. We further put $\sigma:=c_{\mu_b}$ and $(\mathfrak{g},L(\sigma)):=L(G,\sigma)$.\linebreak
	Considering the $\sigma$-invariant connected integral subgroup $H:=\langle\exp(\mathfrak{h})\rangle$ of $G$ with $L(\sigma)$-invariant Lie algebra $\mathfrak{h}:=\overline{[\mathfrak{n},\mathfrak{n}]}\oplus \mathfrak{n}\leq \mathfrak{g}_+\oplus\mathfrak{g}_-$ (cf.\ Lemma~\ref{lem:[n,n]+n} and Lemma~\ref{lem:integralSubspaceOfSymLieGroup}), we obtain a morphism $\iota\colon (H,\sigma_H,H_b)\rightarrow (G,\sigma,G_b)$ of symmetric pairs with $\sigma_H:=\sigma|_H^H$ and $H_b:=G_b\cap H$ (cf.\ Lemma~\ref{lem:integralSubspaceOfSymLieGroup} and Example~\ref{ex:integralSubspaceOfHomogeneousSpace}).
	It induces a connected integral subspace\linebreak $\Sym(\iota)\colon H/H_b\rightarrow G/G_b$  with $\Lts(H/H_b)=\mathfrak{h}_-=\mathfrak{n}$ by Example~\ref{ex:integralSubspaceOfHomogeneousSpace}.
	From $H/H_b=\langle\Exp_{H/H_b}(\mathfrak{n})\rangle$, we obtain $H/H_b = \langle\Exp_{G/G_b}(\mathfrak{n})\rangle$ when considering $H/H_b$ as a subset of $G/G_b$.
	The uniqueness assertion follows by Lem\-ma~\ref{lem:uniquenessOfIntegralSubspaces}, since $N$ has to be connected (and hence $\Inn(N)$-transitive) because of $N_0=\langle \Exp_{(M,b)}(\mathfrak{n})\rangle$.	
\end{proof}
\begin{corollary}\label{cor:1-to-1CorrespondenceIntegralSubspaces}
	The connected integral subspaces of a pointed symmetric space $(M,b)$ are in one-to-one correspondence with the closed triple subsystem of $\Lts(M,b)$.
\end{corollary}
\begin{proposition} \label{prop:LtsOfIntegralSubpace}
	Let $(M,b)$ be a pointed symmetric space and $(N,b)$ a connected integral subspace. If its Lie triple system $\mathfrak{n}\leq \Lts(M,b)$ is separable, i.e., if it contains a countable dense subset, then we have
	$\mathfrak{n}=\{x\in \Lts(M,b)\colon \Exp_{(M,b)}(\RR x)\subseteq N\}$.
\end{proposition}
\begin{proof}
	W.l.o.g., we assume $M$ to be connected. Having the situation of Proposition~\ref{prop:integralSubreflectionSpace}, we refer, in the following, to the definitions and notation made in its proof.
	Since $\mathfrak{n}$ is separable, the Banach--Lie algebra $\mathfrak{h}=\overline{[\mathfrak{n},\mathfrak{n}]}\oplus \mathfrak{n}$ is also separable, so that
	$\mathfrak{h}=\{x\in \mathfrak{g}\colon \exp(\RR x)\subseteq H\}$
	(cf.\ Section~\ref{sec:terminologySub}).
	
	Given any $x\in \Lts(G/G_b)=\mathfrak{g}_-$ with $\Exp_{G/G_b}(\RR x)\subseteq \Sym(\iota)(H/H_b)$, we shall show that $x\in\mathfrak{n}=\mathfrak{h}_-$. For any $t\in\RR$, we have
	$$\Exp_{G/G_b}(tx) \ =\ \exp(tx)G_b \ \in\ \Sym(\iota)(H/H_b),$$
	i.e., there exist some $h\in H$ and some $k\in G_b$ satisfying $\exp(tx)=hk$.
	Since we have $\sigma(\exp(tx))=\exp(L(\sigma)(tx))=\exp(-tx)$, we obtain
	$$\exp(tx) \ =\ \big(\sigma(\exp(tx))\big)^{-1} \ =\ \sigma(k)^{-1}\sigma(h)^{-1} \ =\ k^{-1}\sigma(h)^{-1},$$
	so that
	$$\exp(2tx) \ =\ \exp(tx)\exp(tx) \ =\ hkk^{-1}\sigma(h)^{-1} \ =\ h\sigma(h)^{-1} \in H.$$
	Therefore, we obtain $\exp(\RR x)\subseteq H$, i.e., $x\in \mathfrak{h}$. Hence it follows that $x\in\mathfrak{g}_-\cap\mathfrak{h}=\mathfrak{h}_-$.
\end{proof}
%
%
%
%
%
\subsection{Symmetric Subspaces} \label{sec:symmetricSubspaces}
A \emph{symmetric subspace} of a symmetric space $M$ is a reflection subspace $N\leq M$ that is a submanifold of $M$. It is clear that such $N$ itself is a symmetric space. If, in addition, $N$ is a split submanifold, then we say that $N$ is a \emph{split symmetric subspace}.
Let $\iota\colon N\rightarrow M$ be the inclusion map. Then, for each $b\in N$, the induced morphism $\Lts_b(\iota)$ is a topological embedding, $N$ being a submanifold. Therefore, a symmetric subspace is an integral subspace and we shall frequently identify $\Lts(N,b)$ with its image $\mathfrak{n}$ in $\Lts(M,b)$ under $\Lts_b(\iota)$.
In the light of (\ref{eqn:imageOfInnAut}), we know for $\Inn(N)$-transitive (e.g.\ connected) $N$ that $N$ splits if and only if $\mathfrak{n}$ splits in $\Lts(M,b)$ as a Banach space.
Every open reflection subspace of a symmetric space is a (split) symmetric subspace. In particular, the connected components of a symmetric space are symmetric subspaces.

\begin{problem}\label{problem:closedness}
	Being a submanifold, a symmetric subspace $N\leq M$ is locally closed. There arises the question whether - or under which conditions - it is moreover closed in $M$.
\end{problem}
\begin{proposition} \label{prop:openSubsymSpace}
	Let $M$ be a symmetric space. Each open symmetric subspace $N\leq M$ is also closed in $M$. Hence, if in addition, $M$ is connected, then $N$ is empty or coincides\linebreak with $M$.
\end{proposition}
\begin{proof}
	If $M$ is connected and $N$ is not empty, then $M$ is generated by $N$ (cf.\ Section~\ref{sec:LtsAndSymSpace}), so that we have $M=\langle N\rangle = N$. Otherwise, we write $M$ as the union $\cup_{i\in I} C_i$ of its connected components. Each non-empty intersection $N\cap C_i$ is an open symmetric subspace of $C_i$ and hence coincides with $C_i$. Denoting by $I_N\subseteq I$ the subset of all indices $i\in I$ for which $N\cap C_i \neq \emptyset$, we obtain
	$N = \cup_{i\in I_N} C_i.$
	As each connected component of $M$ is open, the complement of $N$ is so, being a union of connected components.
\end{proof}

\begin{proposition} \label{prop:LtsOfSubsymspace}
	Let $(M,b)$ be a pointed symmetric space and $(N,b)$ be a symmetric subspace with inclusion map $\iota\colon (N,b)\rightarrow (M,b)$. The image of $\Lts(N,b)$ under $\Lts(\iota)$ is given by
	$\mathfrak{n}:= \{x\in\Lts(M,b)\colon \Exp_{(M,b)}(\RR x)\subseteq N\}$.
\end{proposition}
\begin{proof}
	The inclusion $\im(\Lts(\iota))\subseteq \mathfrak{n}$ is clear, since we have
	$$\Exp_{(M,b)}(\RR\Lts(\iota)(x)) \ =\ \iota(\Exp_{(N,b)}(\RR x)) \ \subseteq \ N$$
	for all $x\in \Lts(N,b)$. We shall prove the converse inclusion: Given any $x\in\mathfrak{n}$, the one-parameter subspace $\alpha\colon\RR\rightarrow M$ with $\alpha(t)=\Exp_{(M,b)}(tx)$ actually lies in $N$, so that we obtain a one-parameter subspace $\alpha_N:=\iota^{-1}\circ\alpha\colon \RR \rightarrow N$. Note that this map is indeed smooth, $N$ being a submanifold. We obtain
	$$x \ =\ \alpha^\prime(0) \ =\ (\iota\circ\alpha_N)^\prime(0) \ =\ \Lts(\iota)(\alpha_N^\prime(0)) \ \in\ \im(\Lts(\iota)),$$
	so that $\mathfrak{n}\subseteq \im(\Lts(\iota))$.
\end{proof}
\begin{proposition} \label{prop:expChartOfSubsymSpace}
	Let $(M,b)$ be a pointed symmetric space and $(N,b)$ be a symmetric subspace with Lie triple system $\mathfrak{n}\leq \Lts(M,b)$. Then there exists an open $0$-neighborhood\linebreak $V\subseteq\Lts(M,b)$ such that $\Exp_{(M,b)}|_V$ is a diffeomorphism onto an open subset of $M$ and 
	$$\Exp_{(M,b)}(V\cap \mathfrak{n}) \ =\ \Exp_{(M,b)}(V) \cap N.$$
\end{proposition}
\begin{proof}
	Let $W\subseteq \Lts(M,b)$ be an open $0$-neighborhood such that $\Exp_{(M,b)}|_W$ is a diffeomorphism onto an open subset of $M$ and let $V_\mathfrak{n}\subseteq \mathfrak{n}$ be an open $0$-neighborhood in $\mathfrak{n}$ such that $\Exp_{(N,b)}|_{V_\mathfrak{n}}=\Exp_{(M,b)}|_{V_\mathfrak{n}}$ is a diffeomorphism onto an open subset of $N$ and $V_\mathfrak{n}\subseteq W$. Since $N$ is a topological subspace of $M$, there exists an open set $U$ in $M$ with $\Exp_{(M,b)}(V_\mathfrak{n}) = U \cap N$. We assume that $U=\Exp_{(M,b)}(V)$ for some open $0$-neighborhood $V\subseteq W$, as we can otherwise replace $U$ by $U\cap \Exp_{(M,b)}(W)$.
	
	It is clear that $\Exp_{(M,b)}(V\cap \mathfrak{n})\subseteq\Exp_{(M,b)}(V) \cap N$, since $\Exp_{(M,b)}(\mathfrak{n})=\Exp_{(N,b)}(\mathfrak{n})\subseteq N$. Conversely, given any $x\in V$ with $\Exp_{(M,b)}(x)\in N$, we have $\Exp_{(M,b)}(x)\in\Exp_{(M,b)}(V_\mathfrak{n})$. As the exponential map is injective on $W$, we obtain $x\in V_\mathfrak{n}$ and hence $x\in\mathfrak{n}$, entailing the converse inclusion.
\end{proof}
\begin{remark}\label{rem:integralSubspaceWithSubspaceTopology}
	For the proof, it actually suffices to let $(N,b)$ be an integral subspace whose topology is induced by $M$, i.e., the inclusion map $\iota\colon N\rightarrow M$ is a topological embedding, but, by the assertion, $N$ is then automatically a submanifold of $M$ and hence a symmetric subspace.
\end{remark}
\begin{corollary}\label{cor:integralSubspaceWithSubspaceTopology}
	Let $M$ be a symmetric space and $N\leq M$ an integral subspace whose topology is induced by $M$, i.e., the inclusion map $\iota\colon N\rightarrow M$ is a topological embedding. Then $N$ is a symmetric subspace of $M$.
\end{corollary}
\begin{proposition} \label{prop:subreflectionSpaceAndLocalSubmanifold}
	Let $(M,b)$ be a pointed symmetric space and $(N,b)\leq (M,b)$ be a reflection subspace that is a local submanifold at $b$.
	If the natural action of $\Aut(M,N):=\{g\in\Aut(M)\colon g(N)=N\}$ on $N$ is transitive (e.g.\ if $N$ is $\Inn(N)$-transitive), then $N$ is a submanifold of $M$, hence a symmetric subspace.
\end{proposition}
\begin{proof}
	Given any $b^\prime\in N$, there exists some $g\in\Aut(M,N)$ with $g(b)=b^\prime$ by assumption.
	Let $\varphi\colon U\rightarrow V\subseteq E$ be a chart of $M$ at $b$ with $\varphi(U\cap N)=V\cap F$ for some closed subspace $F$ of $E$. 
	Then the chart $\varphi_g\colon g(U)\rightarrow V\subseteq E$ at $b^\prime$ with $\varphi_g:=\varphi\circ g^{-1}$ satisfies $\varphi_g(g(U)\cap N) = V\cap F$, so that $N$ is a local submanifold at $b^\prime$ and hence a submanifold of $M$.
\end{proof}
\begin{lemma} \label{lem:ExF->M,(x,y)->Exp(x)Exp(y)}
	Let $(M,b)$ be a pointed symmetric space and let $E$ and $F$ be closed subspaces of $\Lts(M,b)$ such that $\Lts(M,b)=E\oplus F$ is a decomposition of $\Lts(M,b)$ as a direct sum of Banach spaces. Then the map
	$\Phi\colon E\times F \rightarrow M$, $(x,y)\mapsto \Exp_{(M,b)}(x)\cdot\Exp_{(M,b)}(y)$
	is a local diffeomorphism around $(0,0)$.
\end{lemma}
\begin{proof}
	The tangent map of $\Phi$ at $(0,0)$ is given by
	\begin{eqnarray*}
		T_{(0,0)}\Phi(v,w) &=& T_{(b,b)}\mu\big(T_0(\Exp_{(M,b)}|_E)(v),T_0(\Exp_{(M,b)}|_F)(w)\big) \\
		&=& T_{(b,b)}\mu(v,w) \ =\ 2v - w
	\end{eqnarray*}
	(cf.\ Section~\ref{sec:LtsAndSymSpace}). 
	Thus it is a topological linear isomorphism, entailing the assertion.
\end{proof}%
\begin{proposition} \label{prop:integralSubSpace=subSymSpace}
	Let $(M,b)$ be a pointed symmetric space and $(N,b)\leq (M,b)$ be an $\Inn(N)$-transitive (e.g.\ connected) integral subspace with Lie triple system $\mathfrak{n}\leq \Lts(M,b)$ that splits as a Banach space. Let $F$ be a complement of $\mathfrak{n}$, i.e., $\Lts(M,b)=\mathfrak{n}\oplus F$.
	Then the following are equivalent:
	\begin{enumerate}
		\item[\rm (a)] $N$ is a symmetric subspace of $M$ (and hence automatically a split symmetric subspace).
		\item[\rm (b)] There exists a $0$-neighborhood $W\subseteq F$ with $N\cap \Exp_{(M,b)}(W)=\{b\}$.
	\end{enumerate}
\end{proposition}
\begin{proof}
	(a)$\Rightarrow$(b): Due to Proposition~\ref{prop:expChartOfSubsymSpace}, there exists an open $0$-neighborhood $V:=V_{\mathfrak{n}}\times V_F\subseteq \mathfrak{n}\oplus F$ such that $\Exp_{(M,b)}|_V$ is a diffeomorphism onto an open subset of $M$ and
	$\Exp_{(M,b)}(V_\mathfrak{n}) \ =\ \Exp_{(M,b)}(V)\cap N$.
	Intersecting both sides of this equation with $\Exp_{(M,b)}(V_F)$ leads to
	$$\{b\} \ =\ \Exp_{(M,b)}(V_\mathfrak{n}\cap V_F) \ =\ \Exp_{(M,b)}(V_F)\cap N.$$
	
	(b)$\Rightarrow$(a): Due to Proposition~\ref{prop:subreflectionSpaceAndLocalSubmanifold}, it suffices to show that $N$ is a local submanifold of $M$ at $b$. By Lemma~\ref{lem:ExF->M,(x,y)->Exp(x)Exp(y)}, there exists an open $0$-neighborhood $V:=V_\mathfrak{n}\times V_F\subseteq \mathfrak{n}\oplus F$ with $V_F\subseteq W$ such that $\Phi\colon \mathfrak{n}\oplus F\rightarrow M$, $x\oplus y\mapsto \Exp_{(M,b)}(x)\cdot\Exp_{(M,b)}(y)$ restricts to a diffeomorphism $\Phi|_V$ onto an open subset of $M$. We claim $\Phi(V_\mathfrak{n})=\Phi(V)\cap N$, which entails that $N$ is a local submanifold at $b$. The inclusion $\Phi(V_\mathfrak{n})\subseteq\Phi(V)\cap N$ is clear, since $\Phi(V_\mathfrak{n})\subseteq \Phi(\mathfrak{n})\subseteq N\cdot b \subseteq N$. To see the converse inclusion, take any $x\oplus y \in V$ with $\Phi(x\oplus y)\in N$ and show $y=0$. Because of $\Exp_{(M,b)}(x)\in\Exp_{(M,b)}(\mathfrak{n})\subseteq N$, we have $\Exp_{(M,b)}(y)=\Exp_{(M,b)}(x)\cdot \Phi(x\oplus y)\in N$ and hence $\Exp_{(M,b)}(y)\in N\cap\Exp_{(M,b)}(V_F)=\{b\}$. Thus we obtain $\Phi(-y)=b\cdot\Exp_{(M,b)}(-y)=\Exp_{(M,b)}(y)=b$, so that $y=0$ follows by the injectivity of $\Phi$.
\end{proof}
\begin{remark} \label{rem:(a)=>(b)WithoutInn(N)-transitivity}
	The proof of (a)$\Rightarrow$(b) does not require the $\Inn(N)$-transitivity of $N$.
\end{remark}

\begin{proposition}[Preimages of symmetric subspaces]
	\label{prop:preimagesOfSubsymspaces}
	Let $f\colon M_1 \rightarrow M_2$ be a morphism of symmetric spaces and $N_2\leq M_2$ a symmetric subspace. Then $N_1:=f^{-1}(N_2)$ is a symmetric subspace of $M_1$ and, for each base point $b_1\in M_1$, we have
	$$\Lts(N_1,b_1)=\Lts_{b_1}(f)^{-1}(\Lts(N_2,b_2))$$
	with $b_2:=f(b_1)$.
\end{proposition}
\begin{proof}
	Being the preimage of a reflection subspace, $N_1$ is a reflection subspace of $M_1$. To see that it is a submanifold (and hence a symmetric subspace), we take any $b_1\in N_1$ and shall show that it is a local submanifold around $b_1$. We put $b_2:=f(b_1)$, $\mathfrak{n}_2:=\Lts(N_2,b_2)$ and $\mathfrak{n}_1:=\Lts_{b_1}(f)^{-1}(\mathfrak{n}_2)$. Due to Proposition~\ref{prop:expChartOfSubsymSpace}, there is an open $0$-neighborhood $V_2\subseteq \Lts(M_2,b_2)$  such that 
	$\Exp_{(M_2,b_2)}|_{V_2}$ is a diffeomorphism onto an open subset of $M_2$ and
	$$\Exp_{(M_2,b_2)}(V_2\cap \mathfrak{n}_2) \ =\ \Exp_{(M_2,b_2)}(V_2) \cap N_2.$$
	We restrict the exponential map of $(M_1,b_1)$ to an open $0$-neighborhood $V_1\subseteq \Lts(M_1,b_1)$ such that $\Exp_{(M_1,b_1)}|_{V_1}$ is a diffeomorphism onto an open subset of $M_1$ and
	$\Lts_{b_1}(f)(V_1)\subseteq V_2$.
	We claim that
	$$\Exp_{(M_1,b_1)}(V_1\cap \mathfrak{n}_1) \ =\  \Exp_{(M_1,b_1)}(V_1) \cap N_1,$$
	which entails that $N_1$ is a local submanifold around $b_1$.
	The inclusion $\Exp_{(M_1,b_1)}(V_1\cap \mathfrak{n}_1) \subseteq  \Exp_{(M_1,b_1)}(V_1) \cap N_1$ is clear, since $\Exp_{(M_1,b_1)}(V_1\cap \mathfrak{n}_1) \subseteq N_1$ follows by
	$$f(\Exp_{(M_1,b_1)}(\mathfrak{n}_1)) \ =\  \Exp_{(M_2,b_2)}(\Lts_{b_1}(f)(\mathfrak{n}_1)) \ \subseteq\ \Exp_{(M_2,b_2)}(\mathfrak{n}_2) \ \subseteq\ N_2.$$
	To see the converse inclusion, take any $x\in V_1$ with $\Exp_{(M_1,b_1)}(x)\in N_1$ and show that $x\in\mathfrak{n}_1$. 
	Because of
	$$\Exp_{(M_2,b_2)}(\Lts_{b_1}(f)(x)) \ =\ f(\Exp_{(M_1,b_1)}(x)) \ \in\ N_2,$$
	we obtain
	$$\Exp_{(M_2,b_2)}(\Lts_{b_1}(f)(x)) \ \in\ \Exp_{(M_2,b_2)}(V_2)\cap N_2 \ =\ \Exp_{(M_2,b_2)}(V_2\cap\mathfrak{n}_2),$$
	so that $\Lts_{b_1}(f)(x)\in \mathfrak{n}_2$, i.e., $x\in \mathfrak{n}_1$.
			
	It remains to prove that $\Lts(N_1,b_1)=\mathfrak{n}_1$.
	This follows immediately by Proposition~\ref{prop:LtsOfSubsymspace} and Lemma~\ref{lem:subsetN_1=f^-1(N_2)}.
\end{proof}
\begin{corollary}\label{cor:kernelOfMorphismOfPointedSymSpaces}
	Let $f\colon (M_1,b_1)\rightarrow (M_2,b_2)$ be a morphism of pointed symmetric spaces. Then its kernel $\ker(f):=f^{-1}(b_2)$ is a closed symmetric subspace of $(M_1,b_1)$ with Lie triple system $\Lts(\ker(f))=\ker(\Lts(f))$.
\end{corollary}
%
%
%
%
%
%
%
%
\subsection{Supplement to the Homogeneity of Connected Symmetric Spaces} \label{sec:supplementHomogeneity}
For some purposes, the automorphism group of a connected symmetric space $M$ is too large and it is useful to consider the group $G(M)$ of displacements. In the finite-dimensional setting, this group is an integral subgroup of $\Aut(M)$ and leads to a further identification $M\cong G(M)/G(M)_b$ (cf.\ \cite{Loo69}). 
In the following, we shall deal with the Banach case.
\begin{definition}\label{def:G'(M,b)}
	Given a pointed connected symmetric space $(M,b)$, we consider the symmetric Lie group $(\Aut(M),c_{\mu_b})$ with Lie algebra $(\Der(M),(\mu_b)_\ast)$. We denote by $G^\prime(M,b):=\langle\exp(\mathfrak{g}^\prime(M,b))\rangle$ the connected integral subgroup of $\Aut(M)$ that belongs to the closed subalgebra $$\mathfrak{g}^\prime(M,b):=\overline{[\Der(M)_-,\Der(M)_-]}\oplus\Der(M)_-$$
	of $\Der(M)$ (cf.\ Lemma~\ref{lem:[n,n]+n}). By Lemma~\ref{lem:integralSubspaceOfSymLieGroup}, $G^\prime(M,b)$ is $c_{\mu_b}$-invariant and $(G^\prime(M,b),\sigma^\prime)$ with $\sigma^\prime:=c_{\mu_b}|_{G^\prime(M,b)}$ is a symmetric Lie group.
\end{definition}
\begin{proposition}\label{prop:M=G'(M,b)/G'(M,b)_b}
	Given a pointed connected symmetric space $(M,b)$, the symmetric pair $(G^\prime(M,b),\sigma^\prime,G^\prime(M,b)_b)$ where $G^\prime(M,b)_b\leq G^\prime(M,b)$ denotes the stabilizer subgroup of $b$ leads to an isomorphism
	$$\Phi^\prime\colon G^\prime(M,b)/G^\prime(M,b)_b\rightarrow M,\,  gG^\prime(M,b)_b\mapsto g(b)$$
	of symmetric spaces.
\end{proposition}
\begin{proof}
	The morphism $\iota\colon (G^\prime(M,b),\sigma^\prime,G^\prime(M,b)_b)\rightarrow (\Aut(M),c_{\mu_b},\Aut(M)_b)$ of symmetric pairs induces an integral subspace $\Sym(\iota)\colon G^\prime(M,b)/G^\prime(M,b)_b\rightarrow \Aut(M)/\Aut(M)_b$ with $\Lts(\Sym(\iota))=\id_{\Der(M)_-}$ (cf.\ Example~\ref{ex:integralSubspaceOfHomogeneousSpace}). Since its image $\im(\Sym(\iota))$ is given by $$\big\langle\Exp_{\Aut(M)/\Aut(M)_b}(\Der(M)_-)\big\rangle \ =\ (\Aut(M)/\Aut(M)_b)_0 \ =\ \Aut(M)/\Aut(M)_b,$$
	it is a diffeomorphism by an argument of uniqueness (cf.\ Proposition~\ref{prop:integralSubreflectionSpace}).
\end{proof}
\begin{corollary}\label{cor:G'(M,b)ActsTransitively}
	The natural smooth action $G^\prime(M,b)\times M\rightarrow M$ is transitive.
\end{corollary}
\begin{remark}\label{rem:G(M)<G'(M,b)<G(M)closure}
	The subgroup $G^\prime(M,b)$ of $\Aut(M)$ satisfies $G(M)\leq G^\prime(M,b)\leq \overline{G(M)}$, where the closure is taken in $\Aut(M)$. Indeed, the group $G(M)$ of displacements is the smallest $c_{\mu_b}$-invariant subgroup of $\Aut(M)$ that acts transitively on $M$ (cf.\ \cite[p.~93]{Loo69}), so that the inclusion $G(M)\subseteq G^\prime(M,b)$ follows. Since for each $\xi\in\Der(M)_-$, its flow map $\flow^{\xi}_1$ is given by $\flow^{\xi}_1=\mu_{\Exp_{(M,b)}(\frac{1}{2}\xi(b))}\circ\mu_b$ (\cite[Th.~5.5.1]{Klo09b}), the exponential map $\exp$ of $\Aut(M)$ maps $\Der(M)_-$ into $G(M)$, so that we have $\exp(\mathfrak{g}^\prime(M,b))\subseteq \overline{G(M)}$ by the Trotter product formula and the commutator formula (cf.\ (\ref{eqn:TrotterProductFormula}) and (\ref{eqn:commutatorFormula})).
\end{remark}
%
%
%
%
%
%
%
\subsection{Quotients}\label{sec:quotients}
An equivalence relation $R\subseteq M\times M$ on a reflection space $M$ that is a reflection subspace of $M\times M$ is called a \emph{congruence relation}. A congruence relation $R$ is just an equivalence relation for which the multiplication map on $M$ induces a multiplication map on the quotient $M/R$ (cf.\ quotient laws for magmas in \cite[I.1.6]{Bou89Algebra}). It can be easily checked that $M/R$ becomes a reflection space.
Given a congruence relation $R$ on $M$, every inner automorphism $g\in\Inn(M)$ maps equivalence classes onto equivalence classes. To see this, one firstly observes that for every symmetry $\mu_z$, $z\in M$, a given pair $(x,y)\in M\times M$ satisfies $(x,y)\in R$ if and only if we have $(\mu_z(x),\mu_z(y))\in R$ and secondly respects that $\Inn(M)$ is generated by the set of all symmetries $\mu_z$, $z\in M$.
Thus, if the group $\Inn(M)$ of inner automorphisms of $M$ acts transitively on $M$ (e.g.\ if $M$ is a connected symmetric space), then a congruence relation is determined by each of its equivalence classes, which we then call \emph{normal reflection subspaces of $M$}, denoted by $N\unlhd M$.
In this section, we characterize closed connected normal reflection subspaces of a connected symmetric space for which the quotient becomes a symmetric space.
\begin{lemma} \label{lem:Nclosed<=>Rclosed}
	Let $M$ be an $\Inn(M)$-transitive (e.g.\ connected) symmetric space, $N\unlhd M$ be a normal reflection subspace and $R\leq M\times M$ be its associated congruence relation on $M$. Then $N$ is a closed subset of $M$ if and only if $R$ is a closed subset of $M\times M$.
\end{lemma}
\begin{proof}
	If $R\subseteq M\times M$ is closed, then $N\subseteq M$ is closed, too, since $N=f^{-1}(R)$ with the morphism $f\colon M\rightarrow M\times M$, $x\mapsto (x,b)$, where we fix some $b\in N$.  Conversely, we assume $N$ to be closed in $M$. Then every equivalence class is closed in $M$, since each one can be obtained as the image of $N$ under an (inner) automorphism of $M$. Given any sequence $(x_n,y_n)_{n\in\NN}$ in $R$ with limit $(x,y)$ in $M\times M$, we shall show that $(x,y)\in R$. Let $$\varphi:=(\Exp_{(M,x)}|_V^U)^{-1}\colon U \rightarrow V\subseteq \Lts(M,x)$$ be a normal chart at $x$. The smooth map
	$$\textstyle\sqrt{\cdot}\colon U\rightarrow U,\ \Exp_{(M,x)}(v)\mapsto \Exp_{(M,x)}(\frac{1}{2}v) \quad \text{(with $v\in V$)}$$
	satisfies $\sqrt{u}\cdot u = x$ and $\sqrt{u}\cdot x = u$ for all $u\in U$ (cf.\ Section~\ref{sec:LtsAndSymSpace}) and further $\sqrt{x}=x$.
	For sufficiently large $n\in\NN$, we have $x_n\in U$, so that
	$$x\cdot y \ =\ \lim_{n\rightarrow\infty}\sqrt{x_n}\cdot \lim_{n\rightarrow\infty}y_n \ =\ \lim_{n\rightarrow\infty}(\sqrt{x_n}\cdot y_n)$$
	by the continuity of the map $\sqrt{\cdot}$ and of the multiplication map. Each $\sqrt{x_n}\cdot y_n$ is obtained in the equivalence class $[\sqrt{x_n}\cdot x_n]=[x]$, entailing that $x\cdot y\in [x]$ by the closedness of the equivalence classes. Hence it follows that $y\in [x\cdot x]=[x]$, i.e., $(x,y)\in R$.
\end{proof}
\begin{proposition}[Lie triple systems of closed normal reflection subspaces]
\label{prop:LtsOfClosedNormalReflectionSubspace}
	Let $(M,b)$ be a pointed $\Inn(M)$-transitive (e.g.\ connected) symmetric space and $(N,b)\unlhd (M,b)$ be a closed normal reflection subspace. Then the Lie triple system $\mathfrak{n}\leq \Lts(M,b)$ of $N$ is a (closed) ideal.
\end{proposition}
\begin{proof}
	Cf.\ \cite[p.~131]{Loo69} for the finite-dimensional case. Let $R\leq M\times M$ be the associated congruence relation on $M$.	We put $\mathfrak{m}:=\Lts(M,b)$ and let $\mathfrak{r}\leq\mathfrak{m}\times\mathfrak{m}$ be the Lie triple system of the closed reflection subspace $R\leq M\times M$ (cf.\ Lemma~\ref{lem:Nclosed<=>Rclosed}), so that $$\mathfrak{r} \ =\ \{(x,y)\in\mathfrak{m}\times\mathfrak{m}\colon (\Exp_{(M,b)}(tx),\Exp_{(M,b)}(ty))\in R \text{ for all $t\in\RR$}\}$$
	by Definition~\ref{def:LtsOfClosedSubreflectionSpace}.
	Since we have $N\times \{b\}=R\cap (M\times \{b\})$, it follows that
	\begin{eqnarray*}
		\mathfrak{n}\times \{0\} &=& \{(x,y)\in\mathfrak{m}\times\mathfrak{m}\colon (\Exp_{(M,b)}(tx),\Exp_{(M,b)}(ty))\in N\times \{b\} \text{ for all $t\in\RR$}\} \\
		&=& \{(x,y)\in\mathfrak{m}\times\mathfrak{m}\colon (\Exp_{(M,b)}(tx),\Exp_{(M,b)}(ty))\in R\cap (M\times \{b\}) \text{ for all $t\in\RR$}\} \\
		&=& \mathfrak{r}\cap (\mathfrak{m}\times \{0\}).
	\end{eqnarray*} 
	To see that $\mathfrak{n}\leq \mathfrak{m}$ is an ideal, take any $x\in\mathfrak{n}$ and $y,z\in\mathfrak{m}$ and show that $[x,y,z]\in\mathfrak{n}$. For this, we firstly observe that $\mathfrak{r}$ contains the diagonal $\Delta_{\mathfrak{m}}$, since $R$ contains the diagonal $\Delta_M$. Therefore, we have
	$$([x,y,z],0) \ =\ ([x,y,z],[0,y,z]) \ =\ [(x,0),(y,y),(z,z)] \ \in\ [\mathfrak{r},\mathfrak{r},\mathfrak{r}] \ \subseteq\ \mathfrak{r}$$
	and hence $([x,y,z],0)\in\mathfrak{r}\cap (\mathfrak{m}\times\{0\})=\mathfrak{n}\times\{0\}$, so that $[x,y,z]\in\mathfrak{n}$.
\end{proof}
We now intend to show that, given a pointed connected symmetric space $(M,b)$ and a closed ideal $\mathfrak{n}\unlhd \Lts(M,b)$, the corresponding connected integral subspace $(N,b)\leq (M,b)$ is normal. Identifying a connected symmetric space with the quotient $G/K$ of a symmetric pair, the idea is to obtain $N$ as the image of a normal subgroup $L\unlhd G$ under the quotient map. This is due to \cite{Loo69} in the finite-dimensional case.  Firstly, we need the following lemmas.
\begin{lemma} \label{lem:imageRofS}
	Let $G$ be a group, $K\leq G$ a subgroup and $q\colon G\rightarrow G/K$, $g\mapsto gK$ the quotient map. Let $L\unlhd G$ be a normal subgroup and $S:=\{(g_1,g_2)\in G\times G\colon g_1^{-1}g_2\in L\}$ the equivalence relation on $G$ that is induced by $L\unlhd G$. The image $R\subseteq G/K \times G/K$ of $S$ under the map $q\times q\colon G\times G\rightarrow G/K \times G/K$ is an equivalence relation on $G/K$. It is given by $R=\{(g_1K,g_2K)\in G/K\times G/K\colon g_1^{-1}g_2\in LK\}$. Every equivalence class of $S$ is mapped by $q$ onto an equivalence class of $R$.
\end{lemma}
\begin{proof}
	Since $L$ is a normal subgroup of $G$, the set $P:=LK=KL$ is a subgroup of $G$. We claim that
	\begin{equation} \label{eqn:P=LK}
		(g_1K,g_2K)\in R \ \Leftrightarrow\ g_1^{-1}g_2\in P \ \Leftrightarrow\ g_2^{-1}g_1\in P
	\end{equation}
	for all $g_1,g_2\in G$. The second equivalence is trivial, $P$ being invariant under inversion. Given any $g_1,g_2\in G$, we have $(g_1K,g_2K)\in R$ if and only if there exists some $k_1,k_2\in K$ with $(g_1k_1,g_2k_2)\in S$, i.e., with $k_1^{-1}g_1^{-1}g_2k_2\in L$ or, equivalently, with $g_1^{-1}g_2\in Lk_1k_2^{-1}$. This holds if and only if $g_1^{-1}g_2\in LK=P$. To see that $R$ is an equivalence relation follows then easily by (\ref{eqn:P=LK}). Given any $g_1\in G$, the equivalence class of $g_1K \in G/K$ is given by\linebreak $\{g_2K\in G/K\colon g_2\in g_1LK\}$, i.e., by $g_1LKK=g_1LK$, and is hence the image of the equivalence class $g_1L$ of $g_1$ under the map $q$.
\end{proof}
\begin{remark} \label{rem:subgroupS}
	Note that $S$ is a subgroup of $G\times G$ (cf.\ \cite[I.4.4, I.1.6]{Bou89Algebra}).
\end{remark}
\begin{lemma}\label{lem:integralSubgroupS}
	Given a connected Banach--Lie group $G$ with Lie algebra $\mathfrak{g}$, let $\mathfrak{l}\unlhd \mathfrak{g}$ be a closed ideal and $L:=\langle \exp_G(\mathfrak{l})\rangle\unlhd G$ the corresponding normal connected integral subgroup. We endow the induced equivalence relation $S:=\{(g_1,g_2)\in G\times G\colon g_1^{-1}g_2\in L\}$ on $G$ with the structure of a smooth Banach manifold making the bijective map $\Phi\colon G\times L \rightarrow S$, $(g,l)\mapsto (g,gl)$ a diffeomorphism. Then $S$ becomes a connected integral subgroup of $G\times G$.
\end{lemma}
\begin{proof}
	It is easy to check that $\Phi$ is a correctly defined map that is bijective with inverse map $\Phi^{-1}\colon S\rightarrow G\times L$, $(g_1,g_2)\mapsto (g_1, g_1^{-1}g_2)$. 
	We organize the proof in three steps.
	
	{\bf Step 1:} \emph{The map $c\colon G\times L\rightarrow L$, $(g,l)\mapsto glg^{-1}$ is smooth.}
	To see this, we firstly show that all partial maps
	$$c_g\colon L\rightarrow L,\ l\mapsto glg^{-1} \quad\text{and}\quad c^l\colon G\rightarrow L,\ g\mapsto glg^{-1}$$
	are smooth. For each $g\in G$, the Lie functor assigns to the conjugation map $C_g\colon G\rightarrow G$, $h\mapsto ghg^{-1}$ the partial map $\Ad_g$ of the adjoint action $\Ad\colon G\times \mathfrak{g} \rightarrow \mathfrak{g}$, so that $C_g\circ \exp_G = \exp_G\circ \Ad_g$. As a closed ideal of $\mathfrak{g}$, the Lie algebra $\mathfrak{l}$ is $\Ad_g$-invariant (cf.\ \cite[p.~26]{Mai62}), so that we obtain $c_g\circ \exp_L = \exp_L \circ \Ad_g|_\mathfrak{l}$, which shows that the homomorphism $c_g$ is smooth. For each $l:=\exp_L(x)\in L$ (with $x\in \mathfrak{l})$, the map $c^l$ is given by $c^l(g)=c_g(\exp_L(x))=\exp_L(\Ad(g,x))$ and is hence smooth. To see that $c^l$ is smooth for all $l\in \langle\im(\exp_L)\rangle = L$, we assume $c^{l_1}$ and $c^{l_2}$ to be smooth for any $l_1,l_2\in L$ and show that $c^{l_1l_2}$ is smooth, but this is clear by observing that $c^{l_1l_2}(g)=gl_1l_2g^{-1}=c^{l_1}(g)c^{l_2}(g)$ for all $g\in G$.
	
	Next, we observe that the map $c$ is locally smooth around $(\eins,\eins)\in G\times L$, since we have $c\circ (\id_G\times \exp_L) = \exp_L\circ\Ad|_{G\times \mathfrak{l}}$. To see that it is globally smooth, take any $(g_0,l_0)\in G\times L$ and observe that
	$$c(g,l) \ =\ glg^{-1} \ =\ (gl_0g^{-1})(gl_0^{-1}lg^{-1}) \ =\ c^{l_0}(g)(c_{g_0}\circ c)(g_0^{-1}g,l_0^{-1}l)$$
	for all $(g,l)\in G\times L$, which shows that $c$ is locally smooth around $(g_0,l_0)$.
	
	{\bf Step 2:} \emph{The group operations in $S$ are smooth.}
	Indeed, the identities
	$$\Phi^{-1}(\Phi(g_1,l_1)\Phi(g_2,l_2)) \ =\ \Phi^{-1}(g_1g_2,g_1l_1g_2l_2) \ =\ (g_1g_2,g_2^{-1}l_1g_2l_2) \ =\ (g_1g_2,c(g_2^{-1},l_1)l_2)$$
	and
	$$\Phi^{-1}((\Phi(g,l))^{-1}) \ =\ \Phi^{-1}(g^{-1},l^{-1}g^{-1}) \ =\ (g^{-1},gl^{-1}g^{-1}) \ =\ (g^{-1},c(g,l^{-1}))$$
	show that multiplication and inversion in $S$ are smooth. Hence, $S$ is a Lie group.
	
	{\bf Step 3:} \emph{The Lie group $S$ is a connected integral subgroup of $G\times G$.} To see this,
	consider the diffeomorphism $\widetilde\Phi\colon G\times G\rightarrow G\times G$, $(g_1,g_2)\mapsto (g_1,g_1g_2)$ and the commutative diagram
	$$\begin{xy}
		\xymatrix{
			G\times L \ar[r]^-{\Phi} \ar@{^{(}->}[d] & S\ar@{^{(}->}[d] \\
			G\times G \ar[r]^-{\widetilde\Phi} & G\times G.
		}
	\end{xy}$$
	We observe that the inclusion maps $S\hookrightarrow G\times G$ and $G\times L\hookrightarrow G\times G$ have the same differentiability properties. Therefore, $S\leq G\times G$ is an integral subgroup, since $L\leq G$ is. Note that $S$ is connected, since $G\times L$ is.
\end{proof}
\begin{lemma} \label{lem:symLieGroup(S,sigma_S)}
	We consider the setting of Lemma~\ref{lem:integralSubgroupS} and let $\sigma$ be an involutive automorphism of $G$ such that $\mathfrak{l}$ is $L(\sigma)$-invariant. Then $S$ is $(\sigma\times\sigma)$-invariant and $(S,\sigma_S)$ with $\sigma_S:=(\sigma\times\sigma)|_S^S$ is a symmetric Lie group.
\end{lemma}
\begin{proof}
	It is clear that $(\sigma\times\sigma)(S)\subseteq S$, since
	$$(g_1,g_2)\in S \ \Rightarrow\ g_1^{-1}g_2\in L \ \Rightarrow\ \sigma(g_1)^{-1}\sigma(g_2)\in \sigma(L)\subseteq L \ \Rightarrow\ (\sigma(g_1),\sigma(g_2))\in S.$$
	To see that $\sigma_S$ is smooth, we observe that
	$$(\Phi^{-1}\circ\sigma_S\circ\Phi)(g,l) \ =\ \Phi^{-1}(\sigma_S(g,gl)) \ =\ \Phi^{-1}(\sigma(g),\sigma(g)\sigma(l)) \ =\ (\sigma(g),\sigma(l)),$$
	so that it is the same to check the smoothness of $\sigma\times \sigma|_L^L$ on $G\times L$, but this follows by Lemma~\ref{lem:integralSubspaceOfSymLieGroup}.
	Therefore, $(S,\sigma_S)$ is a symmetric Lie group.
\end{proof}
The following proposition gives a bridge to gain normal reflection subspaces from normal subgroups via the quotient map given by a symmetric pair.
\begin{proposition} \label{prop:normalN=q(L)}
	Let $(G,\sigma,K)$ be a symmetric pair with connected $G$ and with quotient map $q\colon G\rightarrow G/K$ and let $(\mathfrak{g},L(\sigma))$ be the symmetric Lie algebra of $(G,\sigma)$ (decomposed as $\mathfrak{g}=\mathfrak{g}_+\oplus\mathfrak{g}_-$). Given an $L(\sigma)$-invariant closed ideal $\mathfrak{l}\unlhd \mathfrak{g}$ and its corresponding $\sigma$-invariant normal integral subgroup $L:=\langle \exp_G(\mathfrak{l})\rangle\unlhd G$ (cf.\ Lemma~\ref{lem:integralSubspaceOfSymLieGroup}), then the closed triple subsystem $\mathfrak{l}_-=\mathfrak{l}\cap\mathfrak{g}_-\leq \mathfrak{g}_-$ is an ideal and its corresponding connected integral subspace $N\leq G/K$ is given by $N=q(L)$. It arises as an equivalence class of the congruence relation $R$ on $G/K$ that is given by
	$$R:=\{(g_1K, g_2K)\in G/K\times G/K\colon g_1^{-1}g_2\in LK\},$$
	so that $N$ is normal by definition.
	Further, $R$ carries a natural structure of a connected integral subspace of $G/K\times G/K$.
\end{proposition}
\begin{proof}	 
	By Lemma~\ref{lem:integralSubspaceOfSymLieGroup}, $(L,\sigma_L)$ with $\sigma_L:=\sigma|_L^L$ is a symmetric Lie group with symmetric Lie algebra $(\mathfrak{l},L(\sigma)|_\mathfrak{l}^\mathfrak{l})$.
	As in Example~\ref{ex:integralSubspaceOfHomogeneousSpace}, we turn the inclusion morphism $\iota\colon (L,\sigma_L)\rightarrow (G,\sigma)$ into the morphism $\iota\colon (L,\sigma_L,K_L)\rightarrow (G,\sigma,K)$ of symmetric pairs with $K_L:=K\cap L$ and obtain a connected integral subspace $\Sym(\iota)\colon L/K_L\rightarrow G/K$ with $\im(\Sym(\iota))=q(L)$.
	Its Lie triple system is given by $\mathfrak{l}_-=\mathfrak{l}\cap \mathfrak{g}_-\hookrightarrow \mathfrak{g}_-$, so that $q(L)=N$. We observe that $\mathfrak{l}_-$ is a (closed) ideal in $\mathfrak{g}_-$: Indeed, since $\mathfrak{l}$ is an ideal in $\mathfrak{g}$, we have
	$$[\mathfrak{l}_-,\mathfrak{g}_-,\mathfrak{g}_-] \ =\ [[\mathfrak{l}_-,\mathfrak{g}_-],\mathfrak{g}_-] \ \subseteq\ [\mathfrak{l},\mathfrak{g}_-] \ \subseteq\ \mathfrak{l}$$
	and hence $[\mathfrak{l}_-,\mathfrak{g}_-,\mathfrak{g}_-]\subseteq \mathfrak{l}\cap\mathfrak{g}_- = \mathfrak{l}_-$.
	
	Let $S:=\{(g_1,g_2)\in G\times G\colon g_1^{-1}g_2\in L\}$ be the equivalence relation on $G$ that is induced by $L$. By Lemma~\ref{lem:integralSubgroupS}, $S$ is a connected integral subgroup of $G\times G$. Further, $(S,\sigma_S)$ with $\sigma_S:=(\sigma\times\sigma)|_S$ is a symmetric Lie group by Lemma~\ref{lem:symLieGroup(S,sigma_S)}. 
	
	As in Example~\ref{ex:integralSubspaceOfHomogeneousSpace}, we turn the inclusion morphism $i\colon (S,\sigma_S) \rightarrow (G,\sigma)$ into the morphism $i\colon (S,\sigma_S, K_S) \rightarrow (G\times G,\sigma\times \sigma, K\times K)$ of symmetric pairs with $K_S:=(K\times K)\cap S$ and obtain a connected integral subspace
	$$\Sym(i)\colon S/K_S\rightarrow (G\times G)/(K\times K)\cong G/K \times G/K$$
	whose image $\Sym(i)$ is given by the congruence relation
	$$R=\{(g_1K, g_2K)\in G/K\times G/K\colon g_1^{-1}g_2\in LK\}$$
	on $G/K$ (cf.\ Lemma~\ref{lem:imageRofS}). Note that $N$ is an equivalence class of $R$, since it is the image of the equivalence class $L$ of $S$ under $q$.
\end{proof}
\begin{proposition}[Integral subspaces corresponding to closed ideals]
\label{prop:subreflectionSpaceCorrToIdeal}
	Let $(M,b)$ be a pointed connected symmetric space and $\mathfrak{n}\unlhd \Lts(M,b)$ be a closed ideal. Then the corresponding connected integral subspace $(N,b)\leq (M,b)$ is normal. Its associated congruence relation $R\leq M\times M$ can be endowed with the structure of a connected integral subspace.
\end{proposition}
\begin{proof}
	Cf.\ \cite[p.~131]{Loo69} for the finite-dimensional case.
	We can identify $(M,b)$ with the quotient $G^\prime(M,b)/G^\prime(M,b)_b$ of the symmetric pair $(G^\prime(M,b),\sigma^\prime,G^\prime(M,b)_b)$ of Proposition~\ref{prop:M=G'(M,b)/G'(M,b)_b}, where $G^\prime(M,b)$ is connected and has Lie algebra $$\mathfrak{g}^\prime(M,b) \ =\ \mathfrak{g}^\prime(M,b)_+\oplus\mathfrak{g}^\prime(M,b)_- \ =\ \overline{[\mathfrak{g}^\prime(M,b)_-,\mathfrak{g}^\prime(M,b)_-]}\oplus \mathfrak{g}^\prime(M,b)_-$$
	(cf.\ Definition~\ref{def:G'(M,b)}). Then $\mathfrak{n}$ is considered as an ideal of $\mathfrak{g}^\prime(M,b)_-=\Lts(G^\prime(M,b)/G^\prime(M,b)_b)$
	and leads to a $L(\sigma^\prime)$-invariant closed ideal $\mathfrak{l}:=\overline{[\mathfrak{g}^\prime(M,b)_-,\mathfrak{n}]}\oplus\mathfrak{n}\unlhd \mathfrak{g}^\prime(M,b)$ (cf.\ \cite[p.~132]{Loo69} or \cite[Th.~7.1]{Jac51}) with $\mathfrak{n}=\mathfrak{l}_-$. Applying Proposition~\ref{prop:normalN=q(L)} leads to the assertion.
\end{proof}
Given a Banach--Lie group $G$ and a normal Lie subgroup $N\unlhd G$, we know by\linebreak \cite[Th.~II.2]{GN03} that the quotient $G/N$ carries a unique Lie group structure with\linebreak Lie algebra $L(G)/L(N)$ such that the quotient map $\pi\colon G\rightarrow G/N$ is smooth and\linebreak $L(\pi)\colon L(G)\rightarrow L(G)/L(N)$ is the natural quotient map.\footnote{Note that $\pi$ is then a ``weak'' submersion in the sense that its tangent map induces at each $x\in G$ a linear quotient map $T_x\pi$.}
We shall use this fact to proof a comparable theorem for symmetric spaces.
\begin{lemma} \label{lem:g([x])=[g(x)];gInClosureInn(M)} 
	Let $M$ be a connected symmetric space and $R\leq M\times M$ be a congruence relation on $M$ that is a closed subset. Then every automorphism $g$ in the closure $\overline{\Inn(M)}$ of the group $\Inn(M)\leq \Aut(M)$ maps equivalence classes onto equivalence classes.
\end{lemma}
\begin{proof}
	Since $\overline{\Inn(M)}$ is a group, it suffices to show $(x,y)\in R \Rightarrow (g(x),g(y))\in R$. Let $(g_n)_{n\in\NN}$ be a sequence in $\Inn(M)$ with limit $g\in\overline{\Inn(M)}$. Since the natural action $\Aut(M)\times M\rightarrow M$ is smooth (cf.\ Section~\ref{sec:AutOfConnectedSymSpace}), we have $g(x)=\lim_{n\rightarrow\infty}g_n(x)$ and $g(y)=\lim_{n\rightarrow\infty}g_n(y)$. Therefore we obtain $(g(x),g(y))\in \overline{R}=R$, since $(g_n(x),g_n(y))\in R$ for all $n\in\NN$.
\end{proof}
\begin{lemma}\label{lem:G'-actionOnM/R}
	Let $(M,b)$ be a pointed connected symmetric space and $R\leq M \times M$\linebreak be a congruence relation on $M$ that is a closed subset. Then the natural transitive action\linebreak $G^\prime(M,b)\times M\rightarrow M$, $(g,x)\mapsto g(x)$ (cf.\ Definition~\ref{def:G'(M,b)} and Corollary~\ref{cor:G'(M,b)ActsTransitively}) induces a transitive action $G^\prime(M,b)\times M/R\rightarrow M/R$, $(g,[x])\mapsto [g(x)]$.
\end{lemma}
\begin{proof}
	This follows immediately by Lemma~\ref{lem:g([x])=[g(x)];gInClosureInn(M)}, since $G^\prime(M,b)\subseteq \overline{G(M)}\subseteq \overline{\Inn(M)}$ (cf.\ Remark~\ref{rem:G(M)<G'(M,b)<G(M)closure}).
\end{proof}
\begin{lemma}\label{lem:ker(psi)+n}
	Let $(\mathfrak{g},\theta)$ be a symmetric Lie algebra with $\mathfrak{g}_+=\overline{[\mathfrak{g}_-,\mathfrak{g}_-]}$. Given a closed ideal $\mathfrak{n}\unlhd \mathfrak{g}_-$, then the representation $\mathfrak{g}_+\rightarrow \gl(\mathfrak{g}_-),\ x\rightarrow [x,\cdot]|_{\mathfrak{g}_-}$ (cf.\ Section~\ref{sec:symLieAlgAndLieGrp}) induces a representation $\psi\colon\mathfrak{g}_+\rightarrow \gl(\mathfrak{g}_-/\mathfrak{n})$.
	The subspace $\mathfrak{l}:=\ker(\psi)\oplus \mathfrak{n}$ of $\mathfrak{g}$ is a $\theta$-invariant closed ideal.
\end{lemma}
\begin{proof}
	We obtain the representation $\psi$ by observing that $[\mathfrak{g}_+,\mathfrak{n}]\subseteq[\overline{[\mathfrak{g}_-,\mathfrak{g}_-],\mathfrak{n}]}\subseteq \mathfrak{n}$.
	From $[[\mathfrak{g}_-,\mathfrak{n}],\mathfrak{g}_-]\subseteq \mathfrak{n}$ we deduce that $[\mathfrak{g}_-,\mathfrak{n}]\subseteq \ker(\psi)$. Thus, to see that $\mathfrak{l}$ is an ideal of $\mathfrak{g}$, it remains to point out that $\ker(\psi)$ is an ideal of $\mathfrak{g}_+$ and that $[\ker(\psi),\mathfrak{g}_-]\subseteq \mathfrak{n}$.
\end{proof}
\begin{lemma} \label{lem:lieSubgroupL}
	Let $(G,\sigma,K)$ be a symmetric pair with connected $G$ and with quotient map $q\colon G\rightarrow G/K$ and let $(\mathfrak{g},\theta)$ be the symmetric Lie algebra of $(G,\sigma)$. Further assume that $\mathfrak{g}_+=\overline{[\mathfrak{g}_-,\mathfrak{g}_-]}$.
	Given a closed ideal $\mathfrak{n}\unlhd \mathfrak{g}_-$ for which the integral subspace $N:=\langle\Exp_{G/K}(\mathfrak{n})\rangle \unlhd G/K$ is a symmetric subspace, we consider the closed ideal $\mathfrak{l}:=\ker(\psi)\oplus\mathfrak{n}$ of $\mathfrak{g}$ (cf.\ Lemma~\ref{lem:ker(psi)+n}). Its corresponding $\sigma$-invariant normal integral subgroup $L:=\langle\exp_G(\mathfrak{l})\rangle$ of $G$ (cf.\  Lemma~\ref{lem:integralSubspaceOfSymLieGroup}) is a Lie subgroup.
\end{lemma}
\begin{proof}
	The adjoint representation $\Ad\colon G\rightarrow \GL(\mathfrak{g})$, $g\mapsto L(c_g)$ (with conjugation map $c_g\colon G\rightarrow G$, $h\mapsto ghg^{-1}$) induces a representation $G^\sigma\rightarrow \GL(\mathfrak{g}_-)$, $g \mapsto L(c_g)|_{\mathfrak{g}_-}$. Indeed, for each $g\in G^\sigma$, we have $\sigma\circ c_g = c_g\circ \sigma$, so that $L(c_g)$ is an automorphism of the symmetric Lie algebra $(\mathfrak{g},\theta)$, which leads to $\calLts(L(c_g))=L(c_g)|_{\mathfrak{g}_-}\in\GL(\mathfrak{g}_-)$.
	Since $\mathfrak{l}$ is an ideal of $\mathfrak{g}$, we have $\Ad(G)(\mathfrak{l})\subseteq \mathfrak{l}$ and hence $\Ad(G^\sigma)(\mathfrak{n})\subseteq \mathfrak{l}\cap\mathfrak{g}_-=\mathfrak{n}$. Thus, we obtain a representation $\Psi\colon G^\sigma\rightarrow \GL(\mathfrak{g}_-/\mathfrak{n})$. It is easy to check that its derived representation $L(\Psi)$ is given by $\psi$, since $L(\Ad)=\ad\colon\mathfrak{g}\rightarrow\gl(\mathfrak{g})$, $x\mapsto [x,\cdot]$. Being the kernel of a representation, $\ker(\Psi)$ is a Lie subgroup of $G^\sigma$ with Lie algebra $\ker(\psi)$.
	We now organize the proof in two steps.
	
	{\bf Step 1:} \emph{Observe that $L\cap G^\sigma \subseteq \ker(\Psi)$.} Since $\mathfrak{l}$ is a closed ideal of $\mathfrak{g}$, we have $\Ad(G)(\mathfrak{l})\subseteq\mathfrak{l}$,\linebreak so that the adjoint representation $\Ad$ induces a representation $\Ad_{\mathfrak{g}/\mathfrak{l}}\colon G\rightarrow \GL(\mathfrak{g}/\mathfrak{l})$. Its derived representation $\ad_{\mathfrak{g}/\mathfrak{l}}:=L(\Ad_{\mathfrak{g}/\mathfrak{l}})\colon \mathfrak{g}\rightarrow \gl(\mathfrak{g}/\mathfrak{l})$ is induced by $\ad\colon \mathfrak{g}\rightarrow\gl(\mathfrak{g})$.
	From $[\mathfrak{l},\mathfrak{g}]\subseteq\mathfrak{l}$ we deduce that $\mathfrak{l}\subseteq\ker(\ad_{\mathfrak{g}/\mathfrak{l}})$, so that the integral subgroup $L\unlhd G$ lies in $\ker(\Ad_{\mathfrak{g}/\mathfrak{l}})$. The ideal $\mathfrak{l}\unlhd\mathfrak{g}$ being $\theta$-invariant, the quotient $\mathfrak{g}/\mathfrak{l}$ decomposes as $\mathfrak{g}/\mathfrak{l}=\mathfrak{g}_+/\mathfrak{l}_+\oplus \mathfrak{g}_-/\mathfrak{l}_-$. Therefore, $\Ad_{\mathfrak{g}/\mathfrak{l}}$ induces the representation $\Psi$, so that $L\cap G^\sigma$ lies in the kernel of $\Psi$.
	
	{\bf Step 2:} \emph{Observe that $L$ is a Lie subgroup of $G$.} Since $\ker(\Psi)$ is a Lie subgroup of $G^\sigma$, there exists an open 0-neighborhood $V_+\subseteq \mathfrak{g}_+$ such that $\exp_G|_{V_+}$ is a diffeomorphism onto an open subset of $G^\sigma$ and
	\begin{equation}\label{eqn:lieSubgroupL_eqn1}
		\exp_G(V_+\cap\ker(\psi)) \ =\ \exp_G(V_+)\cap \ker(\Psi).
	\end{equation}
	Similarly, since $N$ is a symmetric subspace of $G/K$, there exists an open 0-neighborhood\linebreak $V_-\subseteq \mathfrak{g}_-$ such that $\Exp_{G/K}$ is a diffeomorphism onto an open subset of $G/K$ and
	\begin{equation}\label{eqn:lieSubgroupL_eqn2}
		\Exp_{G/K}(V_-\cap\mathfrak{n})=\Exp_{G/K}(V_-)\cap N.
	\end{equation}
	(cf.\ Proposition~\ref{prop:expChartOfSubsymSpace}). By the Inverse Mapping Theorem, we can shrink $V_+$ and $V_-$ such that the map
	$$\Phi\colon \mathfrak{g}_+\oplus\mathfrak{g}_- \rightarrow G,\ x\oplus y\mapsto \exp_G(y)\exp_G(x)$$
	restricts to a diffeomorphism $\Phi|_V$ (with $V:=V_+\oplus V_-\subseteq \mathfrak{g}$) onto an open subset of $G$. To see that $L$ is a Lie subgroup of $G$, it suffices to show that $\Phi(V\cap\mathfrak{l}) = \Phi(V)\cap L$. The inclusion $\Phi(V\cap\mathfrak{l})\subseteq \Phi(V)\cap L$ is clear. To see the converse inclusion, take any $x\oplus y\in V_+\oplus V_-$ with $\Phi(x\oplus y)\in L$ and show that $x\oplus y \in \ker(\psi)\oplus \mathfrak{n} = \mathfrak{l}$.
	Because of $\exp_G(x)\in (G^\sigma)_0\subseteq K$, we have $q(\Phi(x\oplus y))=q(\exp_G(y))=\Exp_{G/K}(y)$. Since we know by Proposition~\ref{prop:normalN=q(L)} that $q(L)=N$, we deduce that $\Exp_{G/K}(y)\in N$ and hence that $y\in\mathfrak{n}$ by (\ref{eqn:lieSubgroupL_eqn2}). It follows that $\exp_G(x)\in L\cap G^\sigma$, since we have $\exp_G(x)=\exp_G(y)^{-1}\Phi(x\oplus y)\in L$. By Step~1 and (\ref{eqn:lieSubgroupL_eqn1}), we obtain that $x\in\ker(\psi)$.
\end{proof}
\begin{lemma} \label{lem:G/L}
	Considering the setting of Lemma~\ref{lem:lieSubgroupL}, the involutive automorphism $\sigma$ of $G$ induces an involutive automorphism $\sigma_{G/L}$ of the quotient Lie group $G/L$ and we obtain a symmetric pair $(G/L,\sigma_{G/L},\pi(K))$, where $\pi\colon G\rightarrow G/L$ denotes the quotient map.
\end{lemma}
\begin{proof}
	Since $L\unlhd G$ is $\sigma$-invariant, there is induced an involutive automorphism $\sigma_{G/L}$ of the abstract group $G/L$. To see that it is smooth, we show that $\sigma_{G/L}\circ \exp_{G/L} = \exp_{G/L}\circ \theta_{\mathfrak{g}/\mathfrak{l}}$, where $\theta_{\mathfrak{g}/\mathfrak{l}}$ is the involutive automorphism of $\mathfrak{g}/\mathfrak{l}$ that is induced by $\theta$. For this, we recall that $L(\pi)\colon \mathfrak{g}\rightarrow \mathfrak{g}/\mathfrak{l}$ is the natural quotient map, so that it suffices to observe that
	\begin{eqnarray*}
		\sigma_{G/L}\circ \exp_{G/L}\circ L(\pi) &=& \sigma_{G/L}\circ \pi\circ \exp_G \ =\ \pi\circ\sigma\circ \exp_G \ =\ \pi\circ \exp_G\circ \theta \\
		&=& \exp_{G/L}\circ L(\pi)\circ\theta \ =\ \exp_{G/L}\circ \theta_{\mathfrak{g}/\mathfrak{l}}\circ L(\pi).
	\end{eqnarray*}
	It is clear that $\pi(K)\subseteq (G/L)^{\sigma_{G/L}}$, since we have $\sigma_{G/L}(\pi(K))=\pi(\sigma(K))\subseteq \pi(K)$. It remains to show the inclusion $((G/L)^{\sigma_{G/L}})_0\subseteq \pi(K)$. For this it suffices, with regard to $(G^\sigma)_0\subseteq K$, to point out that $L(\pi)(\mathfrak{g}_+)=(\mathfrak{g}/\mathfrak{l})_+$ leads to $\pi((G^\sigma)_0)=((G/L)^{\sigma_{G/L}})_0$ (cf.\ p.~\pageref{page:L(f)(g1_+)=g2_+}).
\end{proof}
\begin{theorem}[Quotient Theorem] \label{th:quotientTheorem}
	Let $M$ be a connected symmetric space, $N\unlhd M$ be a closed connected normal reflection subspace and $R\leq M\times M$ be its associated congruence relation on $M$. Then the following are equivalent:
	\begin{enumerate} 
		\item[\rm (a)] There exists a morphism $f\colon M\rightarrow M^\prime$ to a symmetric space $M^\prime$ such that $R$ is its kernel relation.
		\item[\rm (b)] The quotient reflection space $M/R$ can be made a symmetric space such that the natural quotient map $\pi\colon M\rightarrow M/R$ is a ``weak'' submersion in the sense that it is smooth and its tangent map induces at each $x\in M$ a linear quotient map $\Lts_x(\pi)$.
		\item[\rm (c)] $N$ is a symmetric subspace of $M$.
	\end{enumerate}
\end{theorem}
\begin{proof}
	The implication (b)$\Rightarrow$(a) is trivial and the implication (a)$\Rightarrow$(c) follows by Corollary~\ref{cor:kernelOfMorphismOfPointedSymSpaces}.
	
	(c)$\Rightarrow$(b): We choose some base point $b\in N$. By Proposition~\ref{prop:LtsOfClosedNormalReflectionSubspace}, the Lie triple system $\mathfrak{n}\leq\Lts(M,b)$ of $(N,b)$ is a closed ideal. We can identify $(M,b)$ with the quotient $G/G_b$ of the symmetric pair $(G,\sigma,G_b):=(G^\prime(M,b),\sigma^\prime,G^\prime(M,b)_b)$ of Proposition~\ref{prop:M=G'(M,b)/G'(M,b)_b}, where $G$ is connected and has Lie algebra $$\mathfrak{g} \ =\ \mathfrak{g}_+\oplus\mathfrak{g}_- \ =\ \overline{[\mathfrak{g}_-,\mathfrak{g}_-]}\oplus \mathfrak{g}_-$$
	(cf.\ Definition~\ref{def:G'(M,b)}). We denote by $q\colon G\rightarrow G/G_b$ the natural quotient map. Having then the setting of Lemma~\ref{lem:lieSubgroupL}, we consider the normal Lie subgroup $L:=\langle\exp_{G}(\mathfrak{l})\rangle$ of $G$ with Lie algebra $\mathfrak{l}:=\ker(\psi)\oplus\mathfrak{n}\unlhd\mathfrak{g}$ (where $\psi\colon\mathfrak{g}_+\rightarrow \gl(\mathfrak{g}_-/\mathfrak{n})$ is the natural representation) and the quotient $(G/L)/(\pi_{G/L}(G_b))$ of the symmetric pair $(G/L,\sigma_{G/L},\pi_{G/L}(G_b))$, where $\pi_{G/L}\colon G\rightarrow G/L$ denotes the natural quotient map (cf.\ Lemma~\ref{lem:G/L}).
	
	Since $R\subseteq M\times M$ is a closed subset (cf.\ Lemma~\ref{lem:Nclosed<=>Rclosed}), the transitive smooth action $G\times M\rightarrow M$, $(g,x)\mapsto g(x)$ induces a transitive action $G\times M/R\rightarrow M/R$, $(g,[x])\mapsto g([x])$ by Lemma~\ref{lem:G'-actionOnM/R}.
	The stabilizer $G_N$ of $N\in M/R$ is given by $q^{-1}(N)$, since the orbit map $G\rightarrow M$, $g\mapsto g(b)$ corresponds to $q\colon G\rightarrow G/G_b$. Since we have $q(L)=N$ (cf.\ Proposition~\ref{prop:normalN=q(L)}), we obtain $G_N=LG_b$. Therefore, there is induced a bijective map $G/(LG_b)\rightarrow M/R$, $gLG_b\mapsto g(N)=[g(b)]$ and hence a bijective map 
	$$\Phi\colon (G/L)/(\pi_{G/L}(G_b)) \rightarrow M/R,\ gLG_b\mapsto [g(b)].$$
	It is an isomorphism of reflection spaces, since we have:
	\begin{eqnarray*}
		\Phi(gLG_b \cdot hLG_b) &=& \Phi(g\sigma(g)^{-1}\sigma(h)LG_b) \ =\ [g\sigma(g)^{-1}\sigma(h)G_b]\\
		 &=& [gG_b \cdot hG_b] \ =\ [gG_b]\cdot [hG_b],
	\end{eqnarray*}
	where we view $M$ as $G/G_b$. Thus $M/R$ inherits the structure of a symmetric space. The quotient map $\pi\colon M\rightarrow M/R$ corresponds to the map $G/G_b\rightarrow (G/L)/(\pi_{G/L}(G_b))$ that is induced by the natural map $G\rightarrow G/L\rightarrow (G/L)/(\pi_{G/L}(G_b))$ and is automatically smooth, since $G\mapsto G/G_b$ is a submersion. It is furthermore a weak submersion, since $\pi_{G/L}\colon G\mapsto G/L$ is a weak submersion and $G/L\rightarrow (G/L)/(\pi_{G/L}(G_b))$ is a submersion.
\end{proof}
\begin{remark}\label{rem:Lts(M/N)=m/n}
	The Lie triple system of the pointed symmetric space $M/N:=(M/R,N)$ is given by $\mathfrak{m}/\mathfrak{n}$, where, having chosen some base point $b\in N$, $\mathfrak{m}$ and $\mathfrak{n}\unlhd\mathfrak{m}$ are the Lie triple systems of $(M,b)$ and $(N,b)$, respectively. Indeed, the kernel $\ker(\pi):=\pi^{-1}(N)$\linebreak of the quotient map $\pi\colon (M,b)\rightarrow M/N$ is the symmetric subspace $(N,b)\unlhd (M,b)$ with Lie triple system $\mathfrak{n}=\ker(\Lts(\pi))$ (cf.\ Corollary~\ref{cor:kernelOfMorphismOfPointedSymSpaces}), so that $\Lts(M/N)$ is isomorphic to\linebreak $\mathfrak{m}/\ker(\Lts(\pi))=\mathfrak{m}/\mathfrak{n}$.
\end{remark}
\begin{remark} \label{rem:M/R withQuotientTopology}
	The symmetric space structure of $M/R$ is compatible with the quotient topology. To see this, it suffices to check that the quotient map $\pi\colon M\rightarrow M/R$ is open, equivalently, that for each $b\in M$, a neighborhood of $b$ is mapped onto a neighborhood of $\pi(b)$. Since we have $\pi\circ\Exp_{(M,b)}=\Exp_{(M/R,\pi(b))}\circ\Lts_b(\pi)$, this follows by the fact that the exponential maps $\Exp_{(M,b)}$ and $\Exp_{(M/R,\pi(b))}$ are local diffeomorphisms and $\Lts_b(\pi)$ is an open map.
\end{remark}

\begin{acknowledgements}
	I am grateful to Karl-Hermann Neeb for his helpful communications and proof reading during my research towards this article. This work was supported by the Technical University of Darmstadt and by the Studienstiftung des deutschen Volkes.
\end{acknowledgements}

\bibliography{paperD}
\bibliographystyle{amsalpha}
\end{document}